\documentclass[reqno,12pt,letterpaper]{amsart}
\usepackage{amsmath,amssymb,amsthm,graphicx,mathrsfs,url}
\usepackage[usenames,dvipsnames]{color}
\usepackage[colorlinks=true,linkcolor=Red   ,citecolor=Green]{hyperref}
\usepackage{floatrow}
\usepackage[usenames,dvipsnames]{xcolor}
\usepackage{tikz} 
\usepackage{slashed}
\usepackage{graphicx}
\usepackage[colorlinks=true]{hyperref}
\usepackage{color}
\usepackage{amsmath,amsfonts}
\usepackage{calrsfs}
\usepackage{amsmath,environ}
\usepackage{floatrow}
\usepackage{autonum}
\usepackage{amssymb}
\usepackage{amsmath, amsthm}
\usepackage{dsfont}
\usepackage{fancybox}
\usepackage{inputenc}
\usepackage{amssymb}
\usetikzlibrary{decorations.pathreplacing}

\DeclareMathAlphabet{\pazocal}{OMS}{zplm}{m}{n}

\numberwithin{equation}{section}

\newtheoremstyle{definition_style}
{15pt}
{10pt}
{\itshape}
{}
{\bfseries}
{ }
{\newline}
{}
\theoremstyle{definition}
\newtheorem{Th}{Theorem}[section]
\newtheorem{Lemma}[Th]{Lemma}

\newtheorem{Def}[Th]{Definition}
\newtheorem{Thm}[Th]{Theorem}

\newtheorem{Lem}[Th]{Lemma}

\newtheorem{Ex}[Th]{Example}

\newtheorem{Cor}[Th]{Corollary}

\title[Large Coupling Limits Beyond Definiteness
]{Large Coupling Convergence Beyond Definiteness
}

\author{Christian Koke}

\NewEnviron{equations}{%
	\begin{equation}\begin{gathered}
			\BODY
	\end{gathered}\end{equation}
}

\newtheorem{prop}{Proposition}[section]

\begin{document}
	
	\vspace*{-5mm}
	
	\maketitle

	\begin{abstract}
		We study convergence of operator families of the form \(A_\beta = A + \beta B \) towards an effective operator defined on $\ker(B)$, as the coupling constant $\beta$ tends to infinity. Crucially, we focus on the setting where neither $A$ nor $B$ can be assumed to be positive- (or negative-) semi-definite.  We are hence outside the classical form-theoretic framework, where results based on Kato’s monotone convergence theorem would be applicable. 
		Thus, instead of form methods, our approach  builds on classical resolvent identities to study convergence of the family $\{A_\beta\}_{\beta}$. Our findings are that: (i) \emph{Strong} resolvent convergence holds (without further spectral assumptions) if $A + \beta B$ is self-adjoint  and the compression of $A$ onto $\ker(B)$ is well behaved.  (ii) Under the more detailed assumption that $0 \in \sigma(B)$  is isolated,  \emph{norm} resolvent convergence can be established even if $A+\beta B$ is merely closed, provided the quasinilpotent part of $B$ at zero vanishes and certain conditions on the interplay of $A$ and $B$ are met. Importantly, if $B$ is not self-adjoint we find that the limit operator not only depends on $\ker(B)$ as a Hilbert space, but crucially also on the precise form of the Riesz projector at $0 \in \sigma(B)$ onto $\ker(B)$.
	\end{abstract}

	\section{Introduction}
	The asymptotic analysis of operator families of the form \( A_\beta = A + \beta B \), where \( A \) is a densely defined non-negative self-adjoint operator and \( B \) is a non-negative perturbation, is by now well established. Arguably most attention has been paid to the fundamental example  of $A_\beta$  being an indexed family of Schrödinger operators, with $A=-\Delta$ the free Hamiltonian and $B=V$ a non-negative interaction potential \cite{demuth1991large, demuth1995schrodinger, demuth1993rate, bruneau2002spectral}. 
	In order to account for point interactions, as they appear in certain physical settings \cite{albeverio2004solvable}, the condition that $B$ be a true operator was subsequently relaxed: Instead of arguing at the level of operators $A, B$, the problem was lifted to the level of (densely defined) quadratic forms $\mathfrak{a}, \mathfrak{b} $. When bounded from below, the form $\mathfrak{a}_{\beta} = \mathfrak{a} +  \beta \mathfrak{b}$ is then canonically associated a densely defined operator $A_\beta$ via the Friedrichs extension. A rich detailed theory of such large coupling perturbations of quadratic forms -- as well as resolvent convergence of their associated operators -- has since been established \cite{fukushima1994dirichlet, stollmann1994stoerungstheorie, brasche2001upper, brasche2005dynkin,  brasche2005resolvent, benamor2008sharp, benamor2007sobolev, belhadjali2011large}. Most of these works take as their starting point the  abstract framework of strongly perturbed quadratic forms $a_\beta = \mathfrak{a} + \beta \mathfrak{b} $ initially developed by Kato  \cite{kato1995perturbation} and later extended by Simon \cite{simon1978canonical} who in their monotone convergence theorems established (generalized) strong resolvent convergence of the operator $A_\beta$ associated to $\mathfrak{a}_\beta$ to a limit operator $\underline{A}$ defined on $\ker(\mathfrak{q})$.
	
	\medskip
	
	 Common among most of these works is the need for the considered quadratic forms $\mathfrak{a}, \mathfrak{b}$ (and hence also their associated operators $A,B$) be positive semi-definite. Surveys into the terrain outside the realm of positive definiteness assumptions have so far been limited: In \cite{belhadjali2023large} e.g. the case where $\mathfrak{a}$ and $-\mathfrak{b}$ (as opposed to $\mathfrak{b}$) are assumed to be non-negative, is treated. The series of works \cite{hempel1995strong, hempel2003magnetic,hempel2006schrodinger} considered the setting of Schrödinger operators with strong magnetic fields, and \cite{barbaroux2019resolvent, Duran2024} considered the example of convergence of two dimensional Dirac operators perturbed by an increasing mass term supported only outside of a domain. In particular, the author is not aware of results on large coupling convergence in the sectorial- or more general closed setting.

\medskip

	In this paper, we initiate a systematic study of \emph{large coupling convergence beyond positivity assumptions}, where neither the perturbing operator \( B \) nor the background operator \( A \) is assumed to be associated with a semi-definite form. In contrast to most prior approaches, we hence avoid form methods altogether and instead directly work at the abstract operator level $A_\beta = A + \beta B$.	
	Under suitable assumptions, we then establish convergence of the family of resolvents $\{(A + \beta B -z)^{-1}\}_\beta$ as $\beta \rightarrow \infty$. A necessary condition for the resolvents in this family to be well-defined, is that the operators in the family  $\{A_\beta\}_\beta$ are closed, at least for sufficiently large $\beta \gg 1$. Outside of lower-boundedness assumptions, there exist two typical settings that guarantee that the sum of two operators $A,B$ is closed  \cite{kato1995perturbation, teschl2014mathematical}, which we will hence also consider here:
	\begin{enumerate}
		\item \textbf{$A$ is bounded  relative to $B$:} If $A$ is bounded with respect to a closed operator $B$, then by definition $\mathcal{D}(B) \subseteq \mathcal{D}(A)$ and  there are constants $a,b \geq 0$ so that 	
	\begin{align}\label{rel_bound}
		\quad \|A\psi\| \leq  a\| B \psi\| + b \|\psi\|, \quad \forall \psi \in \mathcal{D}(B).
	\end{align}
	Clearly then, $A$ is also relatively bounded with respect to $\beta B$ and by rescaling $\beta$ if necessary, we may assume without loss of generality, that  $a < 1$ as soon as  (\ref{rel_bound}) is true for \textit{any} generic $a > 0$. Standard results (e.g. \cite[Theorem IV.1.1]{kato1995perturbation}) then 
	guarantee that $A + \beta B$ is closed on $\mathcal{D}(B)$.
\item
	\textbf{$B$ is bounded  and $A$ is closed:} Clearly adding a bounded operator to $(A,\mathcal{D}(A))$ yields an operator that is closed on the same domain.
\end{enumerate}
In these settings, we then  determine conditions under which the resolvent family $\{(A + \beta B -z)^{-1}\}_\beta$ converges to the resolvent of a limit operator $\underline{A}$ densely defined on $\ker{B}$:

In Section \ref{self_adjoint_section} we consider the setting where both $A$ and $B$ are still self-adjoint, but no longer positive- (or negative-) semi-definite. For self-adjoint $B$, the spectral projection $P_{\{0\}}$ then provides a canonical projection onto $\ker{B}$, and we show that as soon as the compression  $\underline{A} = PA$ on $\mathcal{D}(A) \cap P \mathcal{A}$ of $A$ onto $\ker{B}$  is well-behaved, we have \emph{strong} resolvent convergence.
In Section \ref{norm_resolvent_convergence}, we then consider the more general non-self-adjoint setting: At a fundamental level, if $B$ is no longer self-adjoint (or more generally normal), we lose access to the Borel functional calculus and only retain the holomorphic one. In order to still define a spectral projection onto $\ker(B)$, we thus need to assume that $0 \in \sigma(B)$ is isolated, as only then the Riesz-projector onto $\{0\}$ is well defined. This is a stronger assumption on the spectrum of $B$ around $0$ than in the preceding section, which explains why it 
already allows to establish the stronger notion of \emph{norm}-resolvent convergence in many settings. Crucially, we find that within this general non-self-adjoint setting, we need to assume that the quasi-nilpotent part of $B$ at $0 \in \sigma(B)$ vanishes, if we want to hope for convergence at all. Additionally, we find that the limit operator $\underline{A}$ towards which $A_\beta$ convergence in a generalized norm resolvent sense not only depends on $\ker{B}$ as a linear space, but also on the way the Riesz projector $P$ of $B$ at $0 \in \sigma(B)$ projects onto $\ker(B)$.
\medskip
 Throughout our paper, we illustrate our results using examples from partial differential equations and abstract operator theory (Examples  \ref{compressed_example_I}, \ref{rel_bdd_ex}, \ref{compressed_example_II}), high energy physics (Examples \ref{weak_interaction_example}, \ref{discretization_example}, \ref{high_energy_example_II})
and 
graph theory (Examples \ref{graph_theory_I}, \ref{main_graph_example}). In fact a main motivation in writing this paper was to develop tools that can be applied to the study of effective descriptions of directed graphs that contain high-connectivity clusters, which has applications to graph based machine learning \cite{koke2024holonets}. We briefly illustrate this point in Example \ref{main_graph_example} and pursue the study of the graph setting in full depth in a companion paper \cite{koke_graph26}, using the tools developed in this note.

	\section{Strong Resolvent Convergence when $A + \beta B$ is self-adjoint}\label{self_adjoint_section}

	In this section, we assume that $B$ is symmetric bounded and the operator $A$ is self adjoint on its dense domain $\mathcal{D}(A) \subseteq \mathcal{H}$. We are then interested in strong generalized resolvent convergence convergence of the family $A_\beta = A + \beta B$. 
	
	As we discuss below, the limit operator  to which this family converges is defined via the  notion of compression of an operator \cite{Stenger1968, Nudelman2011, Arlinskiy2025}, which we hence briefly revisit here: Given a linear operator $A$ densely defined on the Hilbert space $\mathscr{H} = \mathscr{V} \oplus  \mathscr{V}^{\perp} $, its compression onto the linear subspace $ \mathscr{V} \oplus  \mathscr{V}^{\perp} $ is the induced operator
	\begin{equation}
		P_{ \mathscr{V}} A : \mathcal{D}(A) \cap \mathscr{V} \longrightarrow  \mathscr{V}.
	\end{equation}
If $A$ is unbounded, it is generically possible that $  \mathcal{D}(A) \cap \mathscr{V}  = \{0\}$, so that in-particular the compressed operator is not densely defined, let alone self-adjont \cite{Arlinskiy2025}. If however for example  $\dim(\mathscr{V}^\perp) = \text{codim}(\mathscr{V}) < \infty$, density of the linear manifold $ \mathcal{D}(A) \cap \mathscr{V} \subseteq  \mathscr{V} $ can be established  \cite{GohbergKrein1957}. Furthermore, if in this case $A$ is self adjoint on $\mathcal{D}(A)$, the compression $P_{ \mathscr{V}} A $  of $A$ is in fact also  self-adjoint on  its domain  $\mathcal{D}(A) \cap \mathscr{V}$ \cite{Stenger1968}. Interestingly, this holds even in the extreme case, where $\mathscr{V}^\perp \cap \mathcal{D}(A) = \{0\}$.   
Beyond finite codimension, well behaved  settings of self-adjoint operators compressed onto subspaces with infinite codimension are easily construced: As a simple example take e.g.  a generic densely defined self adjoint operator $(A, \mathscr{D}(A))$ on the Hilbert space $\mathscr{H}$. For some $d \geq 2$ and a a Hermitian Matrix $\mathbf{M}\in \mathds{C}^{d \times d}$, consider then the operator
\begin{equation}
\mathbf{M} \otimes A: \mathds{C}^d \otimes \mathcal{D}(A) \longrightarrow \mathds{C}^d \otimes \mathscr{H},
\end{equation} 
 on the Hilbert space $\tilde{\mathscr{H}} = \mathds{C}^d \otimes \mathscr{H}$.
It is not hard to verify that $\mathbf{M} \otimes A$ is self-adjoint on its domain. Set now $\mathcal{V} = V \otimes \mathcal{H}$ for some linear subspace $V \subseteq \mathds{C}^{d}$. Then if $V$ is a true subspace of $\mathds{C}^d$, we have $\text{codim}(\mathcal{V}) = \infty$ and the compression of $\mathbf{M} \otimes A$ onto $\mathcal{V}$ is given by 
\begin{equation}
	[P_V\mathbf{M} P_V] \otimes A: V \otimes \mathcal{D}(A) \longrightarrow V  \otimes \mathscr{H},
\end{equation} 
 with the orthogonal projection $P_V$ onto $V$ in $\mathds{C}^d$.

We are now interested in the general setting where $\mathcal{V}$ is the kernel of an operator $B$, and the compression of $A$ onto $\mathcal{V} = \ker(B)$ is self adjoint. In this setting, we can establish generalized strong resolvent convergence of the family $A_\beta = A + \beta B$ to an operator $\underline{A}$, exactly given as the compression of $A$ onto $V$:

	\begin{Thm}\label{convergence_to_compression}
		Let \( A: \mathcal{D}(A) \rightarrow \mathcal{H} \) be a densely defined, self-adjoint operator, and let $B$ be self-adjoint and either bounded or so that $A$ is relatively bounded with respect to $B$. Let  \(P := P_{\ker(B)} \equiv \chi_{\{0\}}(B) \)
be the orthogonal projection onto $\ker(B)$.
 Assume that the compression of $\underline{A}$ onto $\ker(B)$ is self-adjoint.
	Then for all \( z \in \mathbb{C} \setminus \mathbb{R} \) and   \( \psi \in \mathcal{H} \):
\begin{align}
	\left\| (A + \beta B - z)^{-1} \psi - P (\underline{A} - z)^{-1} P \psi \right\| \longrightarrow 0 \quad \text{as } \beta \to \infty
\end{align}
	
	\end{Thm}
	\begin{proof}
		The proof proceeds by first establishing \textit{weak} generalized resolvent convergence and then making use of the first resolvent formula to elevate this to strong resolvent convergence. Let $z \in \mathds{C}\setminus\mathbb{R}$ be arbitrary. 
		We begin by establishing that with $Q = Id_{\mathcal{H}} - P$ we have $Q(A + \beta B-z)^{-1}\rightarrow 0$ weakly. We will make use of the fact that since $(A + \beta B -z)(A + \beta B -z)^{-1} \psi = \psi$, we have
		$
		(A - z)^{-1}
		B (A + \beta B - z)^{-1} \psi = \frac{1}{\beta}  \left[	(A - z)^{-1} - (A + \beta B - z)^{-1}\right] \psi.
			$
		Let us take $\eta \in \text{Ran}\{B\} = \text{Ran}\{Q\} = \text{Ran}\{P\}^\perp$. Choose $\xi \in \text{Ran}\{B\}$ such that $B \xi = \eta$.  Then we have 
	$
		\langle \eta, (A + \beta B -z)^{-1} \psi \rangle = \langle B\xi, (A + \beta B -z)^{-1} \psi \rangle = \langle \xi, B(A + \beta B -z)^{-1} \psi \rangle.
		$
		Fix $\epsilon >0$. Since $\text{Ran}\{(A-\overline{z})^{-1}\} = \mathcal{H}$,
we may find $\zeta \in \mathcal{H}$ so that 
	$
	(A-\overline{z})^{-1} \zeta = \xi.
	$
		Thus:
		\begin{align}
		&\left| \langle \eta,  (A + \beta B -z)^{-1} \psi\rangle \right| = \left| \langle \xi, B (A + \beta B -z)^{-1} \psi\rangle \right| \\
		=
		&	\left| \langle (A-\overline{z})^{-1} \zeta, B (A + \beta B -z)^{-1} \psi\rangle \right| 
		= 
\left| \langle \zeta,  \left[	(A - z)^{-1} - (A + \beta B - z)^{-1}\right]	\psi\rangle \right|/\beta
	\rightarrow 0.
		\end{align}
	 Since $z \in \mathds{C}\setminus\mathds{R}$  above was arbitrary, and $(A + \beta B -z)^{-1} Q = \left[Q (A + \beta B -\overline{z})^{-1}  \right]^*$, this also proves that $(A + \beta B -z)^{-1} Q \rightarrow 0$ weakly.
		Next we prove that $(A + \beta B -z)^{-1} P \rightarrow (\underline{A} - z)^{-1}P$ weakly. We recall that we assume that $PA:\mathcal{D}(A) \rightarrow P\mathcal{H} \rightarrow P\mathcal{H}$ is  self adjoint. Thus $(PA - z)\mathcal{D}(A)  = P\mathcal{H}$. Fix $\phi \in \mathcal{D}(A) \rightarrow P\mathcal{H}$ and set $\psi = P(A - z)\phi$.
We clearly have 
	$
			(\underline{A}-z)^{-1} \psi = (\underline{A}-z)^{-1}  P(A - z) \phi = \phi.
	$
		We also find for all $ \eta \in \mathcal{H}$:
		\begin{align}
		&\lim\limits_{\beta \rightarrow \infty} \langle \eta, (A + \beta B - z)^{-1} \psi \rangle	
		= \lim\limits_{\beta \rightarrow \infty} \langle \eta, (A + \beta B - z)^{-1} P(A - z) \phi \rangle\\
		=& \lim\limits_{\beta \rightarrow \infty} \langle \eta, (A + \beta B - z)^{-1} (P + Q)(A - z) \phi \rangle
		= \lim\limits_{\beta \rightarrow \infty} \langle \eta, (A + \beta B - \overline{z})^{-1} (A - z) \phi \rangle\\
		=& \lim\limits_{\beta \rightarrow \infty} \langle \eta, (A + \beta B - \overline{z})^{-1} (A + \beta B - z) \phi \rangle
		= \langle \eta,  \phi \rangle
		\end{align}
		In total we have thus proved the weak convergence 
		$
		(A + \beta B - z)^{-1} \rightarrow P (\underline{A} - z)^{-1} P
		$.
		Next we want to promote this to strong resolvent convergence. To achieve this, we make use of the first resolvent formula
		$
	(A - z)^{-1}  - 	(A - y)^{-1} = (z-y) 	(A - z)^{-1} 	(A - y)^{-1}
	$:

		\begin{align}
&	\left\|
\left[
(A + \beta B - z)^{-1} - P (\underline{A}  - z)^{-1} P 
\right] \psi
\right\|^2\\
=&   \langle  \psi, \frac{1}{\overline{z} - z} \left[ (A + \beta B - \overline{z})^{-1}  - (A + \beta B - z)^{-1} \right]\psi \rangle\\
 +& 
\langle  \psi, \frac{1}{\overline{z} - z} P \left[ (\underline{A}  - \overline{z})^{-1} -  (\underline{A}  - z)^{-1} \right] P \psi \rangle \\
-&
\langle (A + \beta B - z)^{-1} \psi, P (\underline{A}  - z)^{-1} P  \psi \rangle 
-
\langle P (\underline{A}  - z)^{-1} P  \psi, (A + \beta B - z)^{-1} \psi \rangle.\\
		\end{align}
	This converges to zero as terms cancel each other in the limit.
	\end{proof}
	In fact, it suffices to assume that the compression of $A$ onto $\ker(B)$ is only \textit{essentially} self-adjoint on a dense subset of $\ker(B)\cap \mathcal{D}(A)$, as our next corollary shows:

	\begin{Cor}
		In the setting of Theorem \ref{convergence_to_compression}, assume that $B$ is bounded, $A$ is closed and that
		 there only is a dense domain $\mathcal{D} \subseteq P\mathcal{H} \cap \mathcal{D}(A) \rightarrow P\mathcal{H}$ such that $(PA, \mathcal{D})$ is essentially self adjoint. Setting $\underline{A} = \overline{PA}$ the claim of the previous Thoerem still holds.
	\end{Cor}
	
\begin{proof}
	To establish $Q(A + \beta B-z)^{-1}\rightarrow 0$ weakly, we note that since $\text{Ran}\{(A-\overline{z})^{-1}\}$ is dense,
	we may find $\zeta \in \mathcal{H}$ so that 
	$
	\left\|(A-\overline{z})^{-1} \zeta - \xi\right\|\leq \epsilon \cdot |\text{Im}(z)|/(\|\psi\|\cdot \|B\|).
	$
	Thus:
	\begin{align}
		&\left| \langle \eta,  (A + \beta B -z)^{-1} \psi\rangle \right| = \left| \langle \xi, B (A + \beta B -z)^{-1} \psi\rangle \right| \\
		\leq 
		&\left| \langle \xi- (A-\overline{z})^{-1} \zeta, B (A + \beta B -z)^{-1} \psi\rangle \right| + 	\left| \langle (A-\overline{z})^{-1} \zeta, B (A + \beta B -z)^{-1} \psi\rangle \right| \\
		\leq &\epsilon + 	\left| \langle (A-\overline{z})^{-1} \zeta, B (A + \beta B -z)^{-1} \psi\rangle \right|\\
		\leq &\epsilon + 	\left| \langle \zeta, (A- z)^{-1}  B (A + \beta B -z)^{-1} \psi\rangle \right|\\
		= &\epsilon + 	\frac{1}{\beta}\left| \langle \zeta,  \left[	(A - z)^{-1} - (A + \beta B - z)^{-1}\right]	\psi\rangle \right|
		\leq 2 \epsilon,
	\end{align}
	where the last inequality holds for $\beta$ sufficiently large.
To establish the weak convergence 	
$
		(A + \lambda B - z)^{-1} \rightarrow P (\underline{A} - z)^{-1} P
		$,
	we note we assume that $PA : \mathcal{D} \rightarrow P\mathcal{H}$ is essentially self adjoint, we know that $\text{Ran}(PA - z) = P\mathcal{H}$. In other words we have the dense inclusion $\text{Ran}(P(A - z)) \subseteq P\mathcal{H}$. Let $\phi \in \mathcal{D}$ and set $\psi = P(A-z)\phi$. Then
	$
		(\underline{A}-z)^{-1} \psi =  (\underline{A}-z)^{-1}  P(A - z) \phi = \phi.
	$
	We also find for any $\eta \in \mathcal{H}$ that 
	\begin{align}
		 &\lim\limits_{\lambda \rightarrow \infty} \langle \eta, (A + \lambda B - z)^{-1} (P + Q)(A - z) \phi \rangle
		= \lim\limits_{\lambda \rightarrow \infty} \langle \eta, (A + \lambda B - z)^{-1} (A - z) \phi \rangle\\
		=& \lim\limits_{\lambda \rightarrow \infty} \langle \eta, (A + \lambda B - z)^{-1} (A + \lambda B - z) \phi \rangle
		= \langle \eta,  \phi \rangle.
	\end{align}
	Thus we have found for $\eta \in \mathcal{H}$ arbitrary and $\psi$ in the dense set $(PA - z) \mathcal{D} \subseteq \mathcal{H}$ that 
$\lim_{\lambda \rightarrow \infty} \langle \eta, (A + \lambda B - z)^{-1}\psi \rangle = 
\langle \eta, P (\underline{A} - z)^{-1} P \psi \rangle$. 
Let us now extend this result to arbitrary $\tilde{\psi} \in P\mathcal{H}$. With 
$\psi \in \text{Ran} P(A - z)$ such that $\| \tilde{\psi} - \psi  \| \leq |\text{Im}(z)| \epsilon/3$:
\begin{align}
	&	\left|
	\langle \eta, (A + \beta B - z)^{-1} \tilde{\psi}  - \langle \eta, P(\underline{A} - z)^{-1}P \tilde{\psi} \rangle
	\right|		\\
	\leq&	\left|
	\langle \eta, (A + \beta B - z)^{-1} [\tilde{\psi} - \psi]  \rangle
	\right|		 
	+
	\left|
	\langle \eta, (\underline{A} - z)^{-1} [\tilde{\psi} - \psi]  \rangle 
	\right|	\\
	+ &
	\left|
	\langle \eta, \left[(A + \beta B - z)^{-1} - P(\underline{A} - z)^{-1}P \right]\psi \rangle
	\right|	\\
	\leq & \frac{2}{3} \epsilon + 	\left|
	\langle \eta, \left[(A + \beta B - z)^{-1} - P(\underline{A} - z)^{-1}P \right]\psi \rangle
	\right|	\\
\end{align}
For sufficiently large $\beta$, this can in total thus be bounded by $\epsilon$.
In total we have thus proved the weak convergence 
$
(A + \beta B - z)^{-1} \rightarrow P (\underline{A} - z)^{-1} P
$.
	 From here on, we may follow the proof of the main theorem  again.
\end{proof}

As an initial example, we consider a self adjoint operator $A$ perturbed by a large but finite-dimensional perturbation $B$:

\begin{Ex}\label{compressed_example_I}
Let $(A, \mathcal{D}(A))$ be an arbitrary selfadjoint operator and let $B$ be a symmetric operator with finite dimensional range. Then by Stenger's Lemma \cite{Stenger1968} the compression of $A$ onto $\ker{B}$ is self-adjoint. Since $B$ is bounded, Theorem \ref{convergence_to_compression} guarantees generalized strong resolvent convergence of the family $A_\beta = A + \beta B$ towards the compression $\underline{A}$ for any $z \in \mathds{C} \setminus \mathds{R}$. Importantly, the range of $B$ need not lie inside the domain of $A$ for the result to hold.
\end{Ex}

As a "real world example" about which Theorem \ref{convergence_to_compression} allows to draw conclusions, we consider a large coupling limit inside the weak sector of the standard model of particle physics:

\begin{Ex}\label{weak_interaction_example}
	In the Standard Model of particle physics, the weak interaction is mediated by an \(\mathrm{SU}(2)\) gauge field
	 \cite{peskin1995introduction, hamilton2017mathematical}. This gauge interaction is responsible for processes such as beta decay and neutrino scattering \cite{Bettini2014, srednicki2007quantum}. In this example, we consider this weak-sector structure in a classical gauge background, and investigate what happens in the limit where the coupling constant $\beta$ mediating the strength of the interaction between Fermions and the gauge fields is taken to infinity.
	   To focus on the underlying conceptual behaviour and reduce superfluous complexity, we here consider the setting of \(1+1\)-dimensional Minkowski space.\footnote{Appendix \ref{real_world_weak_example} contains the analogous but technically slightly more involved discussion of the $1+3$-dimensional real world setting.} In this setting, the spinor and gauge structure remain nontrivial, but the technical complexity is reduced.

	Mathematically speaking, we here hence consider a system of spinor fields on Minkowski space \(\mathbb{R}^{1+1}\), interacting with a fixed, classical \(\mathrm{SU}(2)\) gauge field. The unknown is a field \(\Psi: \mathbb{R}^{1+1} \to \mathbb{C}^4\), which decomposes as a doublet of Dirac spinors:
	\[
	\Psi = \begin{pmatrix}
		\psi^{(1)} \
		\psi^{(2)}
	\end{pmatrix}^\top, \quad
	\psi^{(a)} = \begin{pmatrix}
		\psi_L^{(a)} \
		\psi_R^{(a)}
	\end{pmatrix}^\top, \quad a = 1, 2,
	\]
	where each \(\psi_L^{(a)}\) and \(\psi_R^{(a)}\) is a complex-valued function of spacetime. These represent the two so called 'chiral' components (\textbf{L}eft and \textbf{R}ight) of a two-component Dirac spinor in \(1+1\) dimensions \cite{hamilton2017mathematical}.  
	The gauge field $W $ maps from spacetime into the Lie algebra $\mathfrak{su}(2)$.
	This is a three dimensional Lie algebra, whose generators $\{\tau^1,\tau^2, \tau^3 \}$ may simply be taken to be given by the Pauli matrices $\{\sigma_x, \sigma_y, \sigma_z \}$ :
		\begin{align}
		\sigma_x&= \begin{pmatrix}
			0 & 1 \\
			1 & 0
		\end{pmatrix}, \quad
	.
		\sigma_y = \begin{pmatrix}
			0 & -i \\
			i & 0
		\end{pmatrix}, \quad
		\sigma_z = \begin{pmatrix}
			1 & 0 \\
			0 & -1
		\end{pmatrix}.
	\end{align}
	Each gauge field  $W_\mu:\mathds{R}^{1+1} \rightarrow \mathfrak{su}(2) $	may thus be written as \(W_\mu = W_\mu^b \frac{\tau^b}{2}\), with the functions \(W_\mu^c\) representing the components of the fixed background gauge field $W_\mu$.
	Crucially, in the weak sector of the standard model, the gauge coupling acts only on the left-chiral components of the spinor field \cite{peskin1995introduction}. The corresponding Hamiltonian (derived in Appendix \ref{1+1_d_weak_interaction_derivation} for the completeness) then takes the following :	
	\begin{align}
		i \frac{\partial}{\partial t} 
		\begin{pmatrix}
			\psi_L \\
			\psi_R
		\end{pmatrix}
		=
		\begin{pmatrix}
			- i \partial_x  \mathbb{I}_2 + g (W_0^c - W_1^c)  \frac{\tau^c}{2}
			& m \mathbb{I}_2 \\
			m \mathbb{I}_2 & i \partial_x  \mathbb{I}_2
		\end{pmatrix}
		\begin{pmatrix}
			\psi_L \\
			\psi_R
		\end{pmatrix}.
	\end{align}
	
	\medskip
	
	To illustrate Theorem~\ref{convergence_to_compression}, we now take each $W^c_\mu \in C^\infty(\mathbb{R}) \cap L^\infty(\mathbb{R})$ to be a smooth, bounded function of the spatial coordinate \(x \in \mathbb{R}\). Then it is easy to check that for any value of the coupling constant \(\beta\), the Hamiltonian \(H_\beta\) is self-adjoint on the dense domain \(\mathbb{C}^4 \otimes H^1(\mathbb{R}) \subseteq \mathbb{C}^4 \otimes L^2(\mathbb{R})\).
For simplicity in presentation (similar results hold in the general case), let us assume that all gauge fields vanish except for \(W_0^3\); i.e. \(W_\mu^a = 0\) unless \(\mu = 0\) and \(a = 3\). Then we may write the \(\beta\)-dependent Hamiltonian \(H_\beta\) as a sum
$ 
	H_\beta = A + \beta B
$
	where the unperturbed and interacting parts are given by
	\begin{align}
		A = \begin{pmatrix}
			- i \partial_x  \mathbb{I}_2 & m \mathbb{I}_2 \\
			m \mathbb{I}_2 & i \partial_x  \mathbb{I}_2
		\end{pmatrix}, \qquad
		B = \frac{1}{2} \begin{pmatrix}
			W_0^3(x) \otimes \tau^3 & 0 \\
			0 & 0
		\end{pmatrix}
		= \frac{W_0^3(x)}{2} \begin{pmatrix}
			I_2 \otimes \tau^3 & 0 \\
			0 & 0
		\end{pmatrix}.
	\end{align}
	Assuming that \(W_0^3(x) \neq 0\) except on a set of measure zero, the kernel of \(B\) consists precisely of the right-handed spinors \(\psi_R \in \mathbb{C}^2 \otimes L^2(\mathbb{R})\). Compressing \(A\) onto the space of right-handed spinors,
$
	\mathscr{V} \equiv 0 \oplus \mathbb{C}^2 \otimes L^2(\mathbb{R}) \subseteq \mathbb{C}^4 \otimes L^2(\mathbb{R}),
$
	yields the compressed operator
$
	\underline{A} = i \partial_x \otimes \mathbb{I}_2,
$
	which is self-adjoint on the domain \(\mathbb{C}^2 \otimes H^1(\mathbb{R})\).
	Thus, 
	Theorem~\ref{convergence_to_compression} establishes strong resolvent convergence of \(A_\beta\) to the effective Hamiltonian
	\[
	\underline{A} = i \partial_x \otimes \mathbb{I}_2.
	\]
	This limiting Hamiltonian describes a \emph{decoupled} model of two massless right-handed Weyl spinors, each governed by a free transport-type evolution equation: As the coupling constant is taken to infinity (\(\beta \to \infty\)), the left-handed degrees of freedom are projected out, and only the right-handed sector remains dynamically relevant.
	
	\medskip
	
	Crucially, we note that the spectrum of \(B\) generically contains both positive and negative values: If we for simplicity e.g. take \(W_0^3 \equiv 1\), then simply \(\sigma(B) = \{\pm \tfrac{1}{2}\}\). Likewise, the spectrum of \(A\) also contains both positive and negative components. In fact, it is readily verified that
$
	\sigma(A) = (-\infty, -m] \cup [m, \infty) 
$ \cite{Thaller1992}.
	Thus,  methods reliant on definiteness of either \(A\) or \(B\) are indeed not applicable here.
	
\end{Ex}

\section{Norm Resolvent convergence when $A + \beta B$ is closed}\label{norm_resolvent_convergence}

\noindent 
In the previous section, we established weak resolvent convergence from first principles and then used the fact that for self-adjoint operators the notions of weak- and strong resolvent convergence coincide. 
In this section, we now want to establish (conditions for) the stronger notion of \emph{norm} resolvent convergence. Since self-adjointness does not aid beyond establishing \emph{strong} resolvent convergence, we drop this assumption and generically consider operators $A$, $B$ that are not necessarily self adjoint.

\medskip 

In Section \ref{self_adjoint_section}, the limit operator was naturally 
defined as a dense operator on the Hilbert space \(\ker(B) \subseteq \mathcal{H}\). 
 While it is reasonable to expect this property to persist beyond the self-adjoint setting, some care has to be taken in determining when exactly convergence can be expected: As an example
to illustrate what can go wrong in the non-self-adjoint setting consider the case of $A_\beta = A + \beta B$, with
\[
A = \begin{pmatrix} 1 & 0 \\ 0 & 1 \end{pmatrix}, \quad B = \begin{pmatrix} 0 & 1 \\ 0 & 0 \end{pmatrix}.
\]
For any \(\beta \in \mathbb{R}\), the operator \( A_\beta - z I = [A + \beta B ]- z I \) is invertible for all \( z \in \mathds{C}\setminus\{1\} \), with the corresponding resolvent given as
\begin{align}\label{norm_example}
	(A_\beta - z I)^{-1} = \frac{1}{(1 - z)^2} \begin{pmatrix} 1 - z & -\beta \\ 0 & 1 - z \end{pmatrix}.
\end{align}
In the self-adjoint setting,  
the resolvents of $A_\beta$ were  bounded as $\|(A + \beta B -z)^{-1}\| \leq 1/|\text{Im}(z)|$ uniformly in $\beta$.
In contrast, here 
 the norm of (\ref{norm_example}) grows without bound as \(\beta \to \infty\). Thus the resolvent family \( (A + \beta B - z I)^{-1} \) in particular may never converge.
Crucially, this pathology stems from the fact that the perturbation $\beta B$ in (\ref{norm_example}) is nilpotent. Indeed, clearly $B^2 =0$ holds true in our example.

\medskip

To allow for convergence instead, we will have to impose that the nilpotent part of \(B\) at zero vanishes.
To formalize this also for potentially unbounded non-self-adjoint operators, recall the Riesz projection \( P_{\{\lambda\}} \) associated to an isolated spectral point \(\lambda \in \sigma(B)\) of a closed, densely defined operator \((B, \mathcal{D}(B))\):
\begin{align}\label{riesz_projector}
	 P_{\{\lambda\}}:= \frac{1}{2 \pi i} \oint_\Gamma (B - z I)^{-1} \, dz.
\end{align}
Here \(\Gamma\) is a positively oriented, closed contour in the resolvent set \(\rho(B)\) encircling only the point \(\lambda\). The projection \( P_{\{\lambda\}}  \) and its complement \( Q_{\{\lambda\}}  := I -  P_{\{\lambda\}} \) satisfy \(  P_{\{\lambda\}} Q_{\{\lambda\}}  = Q_{\{\lambda\}}   P_{\{\lambda\}} = 0 \) and commute with the resolvent of \( B \) \cite{HislopSigal1996}. Moreover, since \( P_{\{\lambda\}}  \) arises via a Bochner integral over an analytic family of bounded operators, its range lies in the domain of \( B \) (endowed with the graph norm), Furthermore, since $B$ commutes with its resolvent, it also commutes with \( P_{\{\lambda\}}  \) on \(\mathcal{D}(B)\).

We next note that for $P = P_{\{0\}} $ the projection onto the isolated point $0 \in \sigma(B)$,  the operator \( P  B \) satisfies \(  \sigma(P  B) = \{0\}\). This -- in turn -- is exactly the definition of a quasi-nilpotent operator \cite{HislopSigal1996, Triolo2024}, as its spectrum reduces to \(  \sigma(P  B) = \{0\}\). This motivates the following definition:
\begin{Def}
	We say the \emph{quasi-nilpotent part} of \( B \) at zero vanishes if
	$
	P_{\{0\}} B = 0.
	$
\end{Def}

As will become clear e.g. from the proof of Theorem \ref{bounded_setting} below, the key to preventing pathological growth of the resolvent lies in ensuring the convergence
\[
- z (\beta B - z I)^{-1} \longrightarrow P_0,
\]
where \((\beta B - z I)^{-1}\) denotes the resolvent of the scaled operator \(\beta B\). The conditions imposed on \( B \) precisely ensure this convergence holds:

\begin{Lem}\label{conv_to_projection}
	Let $B: \mathcal{D}(B) \rightarrow \mathcal{H}$ be closed  with  isolated eigenvalue $0\in \sigma(B)$  for which the quasi-nilpotent part of $B$ vanishes. Then $ \exists C \geq 0$ so that for any $z \neq 0$ we have for sufficienly large $\beta \gg 1$ that $\beta B - z$ is boundedly invertible. Furthermore we have  the following  convergence  towards the Riesz projector $P = P_{\{0\}} $:
	\begin{align}\label{convergence_to_projection}
		\left\|(\beta B - z)^{-1} - P/(-z) \right\| \leq C/\beta \longrightarrow 0.
	\end{align}
	If $B P \neq 0$, we have $\|(\beta B - z)^{-1} \| \rightarrow \infty$ instead.
\end{Lem}
\begin{proof}
	First, by assumption, since $0 \in \sigma(B)$ is isolated, 
	there exists some $\epsilon > 0$ such that the punctured disk
	$
	\mathbb{D}_\epsilon \setminus \{0\} = \{ \lambda \in \mathbb{C} : 0 < |\lambda| < \epsilon \}
	$
	lies entirely within the resolvent set $\rho(B)$ of $B$. Bounded invertibility of $\beta B - z$ is equivalent to bounded invertibility of
	$
	B - z/\beta
	$
	and for $|z| \leq \beta \epsilon$, $ B - z/\beta$ is boundedly invertible.

	To prove the second claim, we leverage the spectral decomposition induced by the Riesz projection $P$:
	On its domain $\mathcal{D}(B)$, we can decompose $B$ into a sum of closed operators as
	$
	B = P B P + Q B Q,
	$
	where $Q = I - P$ is the complementary projection. 
	Since $PQ = QP = 0$, this decomposition induces a corresponding block decomposition of the resolvent as
	$
	(\beta B - z)^{-1} = P (\beta P B P - z)^{-1} P + Q (\beta Q B Q - z)^{-1} Q.
	$
	By assumption, the quasi-nilpotent part of $B$ vanishes, so $P B P = 0$. Hence,
	\[
	P (\beta P B P - z)^{-1} P = P (-z)^{-1} P = - P/z.
	\]
	
	Therefore
	$
	-z (\beta B - z)^{-1} = P + (-z) Q (\beta Q B Q - z)^{-1} Q.
	$	
	To conclude the proof, it suffices to show that there exists $C \geq 0$ such that
	$
	\left\| Q (\beta Q B Q - z)^{-1} Q \right\| \leq C/\beta
	$
	for sufficiently large $\beta$.
	To verify this estimate, note that
	$
	Q (\beta Q B Q - z)^{-1} Q = \frac{1}{\beta} Q \left( Q B Q - \frac{z}{\beta} \right)^{-1} Q.
	$
	Since $0$ is isolated in $\sigma(B)$, and $P$ projects onto the eigenspace of $0$, the operator $Q B Q$ has a resolvent set containing a neighborhood of $0$. In particular, there exists $\epsilon > 0$ such that the closed disk of radius $\epsilon$ centered at $0$ is contained in the resolvent set of $Q B Q$.
	For $|\frac{z}{\beta}| < \epsilon$ we can thus -- by analyticity of the resolvent -- expand it in a Neumann series:
	\[
	Q \left( Q B Q - \frac{z}{\beta} \right)^{-1} Q = \sum_{n=0}^\infty \left( \frac{z}{\beta} \right)^n \left[ Q \left( Q B Q \right)^{-1} \right]^{n+1} Q,
	\]
	Thus,
		$
		\left\| Q (\beta Q B Q - z)^{-1} Q \right\| 
		\leq \frac{1}{\beta} \left\| Q (Q B Q)^{-1} \right\| \sum_{n=0}^\infty \left| z/\beta \right|^n \left\| Q (Q B Q)^{-1} \right\|^n$.
	For sufficiently large $\beta$, the geometric series converges and is bounded by $2$.
	Thus we find
	$
	\left\| Q (\beta Q B Q - z)^{-1} Q \right\| \leq 2 \left\| Q (Q B Q)^{-1} \right\|/\beta.
	$
	Setting
	$
	C := 2 \left\| Q (Q B Q)^{-1} \right\|
	$
	completes the proof.

	It remains to prove that whenever $B P \neq 0$, we have $\|(\beta B - z)^{-1} \| \rightarrow \infty$ instead. We note $\|(\beta B - z)^{-1} \| \geq [ \|(\beta B - z)^{-1}P \| /\|P\|= \|(\beta BP - z)^{-1}P \|/\|P\| $. Suppose $BP \neq 0$ and take $\psi \in \mathcal{D}(B)$ so that $\|P\psi\| = 1$ and $BP\psi \neq 0$. Set $\eta_{\beta,z} = (\beta BP -z)^{-1} P\psi$. Then $B P\psi = (\eta_{\beta,z} + z \psi)/\beta$. Thus we find $\|\eta_{\beta,z} \|/\beta \rightarrow  \|B P\psi\| \neq 0$ as $\beta \rightarrow \infty$ and hence $\|\eta_{\beta,z}\| \rightarrow \infty$. Thus, we have $\|(\beta B - z)^{-1} \| \gtrsim \|(\beta BP - z)^{-1}P \| \geq \|\eta_{\beta,z}\| \rightarrow \infty$.
	
\end{proof}

In the non-selfadjoint setting (or more generally the non-normal setting), we have to assume that $0 \in \sigma(B)$ is isolated, in order to be able to define the Riesz Projector $P_0$: Otherwise the interior of any path $\Gamma \subseteq \rho(B)$ encircling zero
 would necessarily also contain other points in the spectrum of $B$.
 However when determining conditions for \emph{norm} resolvent convergence, assuming that $0 \in \sigma(B)$ is isolated is not a strong limitation in the self-adjoint setting: If $B$ is self-adjoint,  $P_0$ can alternatively  be defined via the Borel functional calculus as $P_0 = \chi_{\{0\}}(B)$. Hence it may be calculated also if $0\in \sigma(B)$ is isolated. As it turns out, if $0 \in \sigma(B)$ is not isolated, we however do not even have \emph{norm} convergence in the simplest setting:
 \begin{prop}
 	If $\beta B$ is self-adjoint and $0 \in \sigma(B)$ is not isolated, we do not have  \emph{norm} convergence of the resolvent:
$
 		\left\|-z(\beta B - z)^{-1} - P_0 \right\| \nrightarrow 0.
$
 \end{prop}
 \begin{proof}
 	This straightforward result follows directly from the spectral theorem and the fact that $-z/(\beta\lambda -z) \rightarrow \chi_{\{0\}}(\lambda)$ pointwise, but not in $L^\infty$.
 \end{proof}
Hence if $0\in \sigma(B)$ is not isolated,  we could not even establish \emph{norm}-resolvent convergence of the family  $A_\beta := 0+ \beta B$ towards the zero operator on $\ker(B)$ (whose resolvent is clearly given as $(0 - z)^{-1} = -P_0/z$ on $\ker(B) = P \mathcal{H}$).

	\subsection{Norm Resolvent Convergence when $A$ is bounded  relative to $B$}\label{nrc}\ \\

We begin by discussing the setting of relative boundedness of $A$ relative to $B$, for which we find the following:
	
\begin{Thm}\label{bounded_setting}
	Let $B: \mathcal{D}(B) \rightarrow \mathcal{H}$ be closed and let $A$ be relatively bounded with respect to $B$. Suppose that zero is an isolated spectral point of \( B \) at which the nilpotent part vanishes. Denote by \( P \) the Riesz projection of \( B \) onto zero. Then $AP \in \mathcal{B}(\mathcal{H})$ and  for any \( z \in \mathbb{C} \) with $
	|z| > \|A P\|, \| P A P\|
	$
	we have \( z \notin \sigma(A + \beta B) \) for all sufficiently large \(\beta\).  Furtheremore there exists a constant \( C\geq 0 \)  such that
	\begin{align}\label{rel_bdd_conv_est}
		\left\| (A + \beta B - z I)^{-1} - P (P A P - z I)^{-1} P \right\| \leq C/\beta \xrightarrow[\beta \to \infty]{} 0.
	\end{align}
	Consequently, the family \( A_\beta := A + \beta B \) converges in the generalized norm resolvent sense to the operator \(\underline{A} := P A P\). 
\end{Thm}
\begin{proof}[Proof of Theorem \ref{bounded_setting}]
	The relative boundedness condition  \ref{conv_to_projection} ensures  that $A (\beta B - z)^{-1} $ is a bounded operator. Since $A P =  -z A (\beta B - z)^{-1} P$, also $AP$ is bounded.
	Thus for any $z$ with $|z| > \|AP\|$, we have that $AP - z$ is invertible by a Neumann series argument. Similarly for  $|z| > \|PAP\|$ and we find
	\begin{align}
		P(AP -z)^{-1} = P \frac1z \sum\limits_{k = 0}^{\infty} \left[ \frac{AP}{z}\right]^k = P \frac1z \sum\limits_{k = 0}^{\infty} \left[ \frac{PAP}{z}\right]^k = P(PAP -z)^{-1}.
	\end{align}
	From this it is also clear that $P(AP -z)^{-1} Q = 0$.
	Our goal is now to approximate $P(AP -z)^{-1} $ with $[-z(\beta B -z)^{-1}](A[-z(\beta B -z)^{-1}] -z)^{-1} $. To this end, we first note
	\begin{align}
		\|  P (AP -z) - (AP -z)^{-1} (-z)(\beta B - z)^{-1}  \| \leq \| (AP - z)^{-1}\| \frac{C |z|}{\beta} \leq \frac{1}{\beta} \frac{C |z|}{|z| - \|AP\|}.
	\end{align}
	Next we note that for $\beta \gg 1$ sufficiently large, we have that $-z A (\beta B -z)^{-1} - z = A P -z + A [-z (\beta B -z)^{-1}- P]$ is boundedly invertible. Indeed, by relative boundedness:
	
	\begin{align}
		\|	A [-z (\beta B -z)^{-1}- P] \| \leq a \| B [-z (\beta B -z)^{-1}- P] \| + b \| [-z (\beta B -z)^{-1}- P]\|.
	\end{align}
	The second term on the right hand side may be estimated using Lemma \ref{conv_to_projection}.
	For the first term, we note that:
	\begin{align}
		&\|B [-z (\beta B -z)^{-1}- P]\| = \|QBQ  (\beta QBQ -zQ)^{-1}Q\| \\
		=&\|Q (Id  -z[QBQ]^{-1}/\beta)^{-1}Q\|/\beta \leq \frac{1}{\beta} \frac{|z|}{1-|z| \cdot \|[QBQ]^{-1}\|/\beta} \leq \frac{2 |z| \|Q\|^2}{\beta}.
	\end{align}
	Thus in total $\|	A [-z (\beta B -z)^{-1}- P] \|  \lesssim 1/\beta$. A Neumann argument then establishes invertibility of $-zA(\beta B -z)^{-1} - z$:
	\begin{align}
		(-zA(\beta B -z)^{-1} - z)^{-1} = (AP -z )^{-1} \sum\limits_{k = 0}^{\infty} [-(AP -z )^{-1} A  [-z (\beta B -z)^{-1}- P]]^k
	\end{align}
	Hence for $\beta \gg 1$ sufficiently large, $-zA(\beta B -z)^{-1} (-zA(\beta B -z)^{-1} - z)^{-1}$ is well defined and
	$
	[A + \beta B -z]\cdot[-z(\beta B -z)^{-1} (-zA(\beta B -z)^{-1} - z)^{-1}] = Id_{\mathcal{H}}
	$.
	Thus $z $ is indeed in the resolvent set of $A + \beta B$. Furthermore we may estimate:
	
	\begin{align}
		&\left\| (A + \beta B - z I)^{-1} - P ( A P - z I)^{-1} P \right\|\\
		=& \| [-z(\beta B -z)^{-1} (-zA(\beta B -z)^{-1} - z)^{-1}] -  P ( A P - z I)^{-1}    \| \\
		\leq &     
		\| [-z(\beta B -z)^{-1} (-zA(\beta B -z)^{-1} - z)^{-1}] -  [-z(\beta B -z)^{-1}( A P - z I)^{-1} ]   \| \\
		+& \| [-z(\beta B -z)^{-1}( A P - z I)^{-1}  -  P ( A P - z I)^{-1}    \|\\
		\lesssim &1/\beta.
	\end{align}
	Thus we have also proved \ref{rel_bdd_conv_est} and hence established generalized norm resolvent convergence.
\end{proof}

In the self adjoint setting, the spectral projection $P = \chi_{\{0\}}(B)$ onto $\ker(B)$ is orthogonal, so $\underline{A}$ is completely determined by $A$ and $\ker(B) \subseteq \mathcal{H}$.
If $B$ is not self-adjoint however, the Riesz projector (\ref{riesz_projector}) is generically not orthogonal.
Crucially, the operator $\underline{A}$ then not only depends on $A$ and the space $\ker(B)$ as a set, but also the Riesz projector $P = P_{\{0\}}$ itself:
\begin{Ex}\label{graph_theory_I}
	Let us consider a Graph $G$ on a vertex set $V = \{1,2,3\}$ consisting of three nodes, with directed edges $E = \{(1,2),(2,1), (2,3), (3,2)\}$. Consider now edge weights $a_{12}=a_{21} = a$ and $b_{2,3} \neq b_{3,2}$.	
	After scaling up the edge weights of edges $(2,3),(3,2)$ by a factor of $\beta \gg 1$, the corresponding (in-)degree Graph Laplacian is given as 
	\begin{equation}
		L_\beta = \begin{pmatrix}
			a & -a &0\\
			-a & a & 0 \\
			0 & 0 & 0
		\end{pmatrix} 
		+
		\beta
		\begin{pmatrix}
			0 & 0 & 0\\
			0 & b_{23} & - b_{23} \\
			0 & -b_{32} & b_{32}
		\end{pmatrix}.
	\end{equation}
	It is not hard to see that in this setting
	\begin{align}
		P = (1,0,0)^{\intercal} (1,0,0)  + \frac{1}{1 + \frac{b_{23}}{b_{32}}} (0,1,1)^\intercal(0,1,b_{23}/b_{32})	
	\end{align}
	Identifying $\ker(B) = P \mathds{R}^3  \cong \mathds{R}^2$ via  $(1,0,0)^{\intercal}\mapsto (1,0)^{\intercal}$ and $(0,1,1)^{\intercal} \mapsto (0,1)^{\intercal}$ we find
	
	\begin{align}
		\underline{L_0} = 
		\begin{pmatrix}
			a & -a \\
			- \frac{a}{1+ \frac{b_{23}}{b_{33}}}  &  \frac{a}{1+ \frac{b_{23}}{b_{33}}}  \\
		\end{pmatrix},
	\end{align}
	which clearly depends not only on the kernel $\ker(B) = \text{span}\{(1,0,0)^{\intercal}, (0,1,1)^{\intercal}\}$ but also on the projection $P$ itself (or equivalently on the left-kernel of $B$ as well).\footnote{We explore this fact and its implications for spectral graph theory in a companion paper \cite{koke_graph26}. This has applications e.g. in graph based machine learning \cite{GraphScatteringBeyond, limitless, koke2024holonets} (see also \cite{koke2023resolvnet, koke2024transferability, koke2025multiscale, koke2025graphscale, koke2025incorporating} for additional work in progress.)}
\end{Ex}

Next let us consider a straightforward example from partial differential equations where both $A$ and $B$ are unbounded:

\begin{Ex}\label{rel_bdd_ex}
Set $\mathcal{H} = L^2(\mathds{R}^d)$ and consider the multiplication operator $A = \frac{1}{|x|}$. By Hardy's inequality, $A$ is relatively bounded with respect to the Laplacian $-\Delta$ if $d \geq 3$ \cite{simon_operator_2015}. Set now $f: [0,\infty) \rightarrow \{0\} \cup  [1, \infty)$ with $f(x) = 0 $ if $0\leq x \leq 1$ and $f(x) = x$ if $x > 1$.  It is not hard to prove that $B$ is also bounded relative to $B = f(\Delta)  $ and that $0 \in \sigma(f(B))$ is isolated with $\ker(B) = \mathcal{F}^{-1}(L^2(B^{d})))$; i.e. those functions whose momentum support is contained within the unit ball $B^{d}$. Thus for any $\alpha \in \mathds{C} \setminus \{0\}$ we have the generalized norm resolvent convergence of $A_\beta = 1/|x| + \beta\cdot \alpha  f(-\Delta)$ towards the integral operator $\underline{A}$ acting on $\psi \in \mathcal{F}^{-1}(L^2(B^{d})))$ as 
\begin{align}
[\underline{A}\psi](x) = \int_{\mathds{R}^d} \frac{J_{\frac{d}{2}}(x - y)}{\|x - y\|^{\frac{d}{2}}\cdot \|y\|} \psi(y)  \, d^dy.
\end{align} 
Here  $J_{\frac{d}{2}}$ denotes the $\frac{d}{2}^{\text{th}}$ Bessel function, which arises as the Fourier transform of the indicator function $\chi_{B^d}(p) = \chi_{B^d}(p^2)$ of the unit ball $B^d$ in momentum space.
	\end{Ex}
	
Of course $A$  need not be  aself-adjoint operator for Theorem \ref{bounded_setting} to apply:	
	\begin{Ex}\label{discretization_example}
Consider the discretized (massless) Dirac Hamiltonian in one spatial dimension (c.f. e.g. \cite{Cornean2023, Koke2016, Koke2020}), which acts on the Hilbert space $\mathds{C}^2 \otimes \ell^2(\mathds{Z})$.
Specifically, we are interested in a discrete approximation using a \textit{non-symmetric} finite difference approximation to the spatial derivative: Replacing the derivative \(\partial_x \psi(x)\) by the forward difference $
	\partial_x \psi_n \approx \psi_{n+1} - \psi_n,
	$
	the Hamiltonian acts on the two-component spinor wavefunction \(\psi_n = (\psi_{n,\uparrow} ,  \psi_{n,\downarrow})^\intercal\) as
	$
	(H_0 \psi)_n = -i \sigma_x (\psi_{n+1} - \psi_n) 
	$,
	with
	 \(\sigma_x
\)  the first Pauli matrix.
	Explicitly, the Hamiltonian acts component-wise as
	
	\[
	(H_0 \psi)_n =
	- i 
	\begin{pmatrix}
		0 & 1 \\
		1 & 0
	\end{pmatrix}
	\left(
	\begin{pmatrix}
		\psi_{n+1, \uparrow} \\
		\psi_{n+1, \downarrow}
	\end{pmatrix}
	-
	\begin{pmatrix}
		\psi_{n, \uparrow} \\
		\psi_{n, \downarrow}
	\end{pmatrix}
	\right)
	.
	\]
	It is not hard to see, that the spectrum of the operator $H$ is given as $\sigma(H_0) = \{ \pm i(e^{ik} -1) : k \in [0, 2 \pi]\}$, so that $H$ (while still bounded and normal) is no longer self adjoint. 
	
	The non-symmetric discretization thus explicitly breaks the Hermiticity of the Hamiltonian, which can be used to circumvent the constraints of the Nielsen– Ninomiya theorem \cite{nakamura2024remarks}: This theorem states that, under mild assumptions including Hermiticity, locality, and discrete translational invariance, lattice fermions must appear in pairs of opposite chirality --- an unwanted phenomenon known as \emph{Fermion doubling}. In practice, this leads to unwanted low-energy doubler modes in symmetric discretizations, complicating the recovery of the correct continuum limit.

	We may now add a potential term $\beta B = \beta V(n)  \mathbb{I}_2$ to the Hamiltonian; say a trapping potential that is zero on a finite number of lattice sites, but satisfies $|V(n)| \rightarrow \infty$ as $n \rightarrow \pm \infty$. Then  $B=V$ is  densely defined and closed on its maximal domain. By Theorem \ref{bounded_setting}, taking the limit $\beta \rightarrow \infty$ leads to an effective Hamiltonian $PH_0P$. 
	Since $V$ is self-adjoint, the Riesz projection $P$ here is simply given as the orthogonal projection onto the space of elements in $ \mathds{C}^2 \otimes \ell^2(\mathds{Z})$ that are supported on lattice sites $\{m \in \mathds{Z}: V(m) = 0\}$, so that $PH_0P$ is a exactly the restriction of $H_0$ to these sites.
\end{Ex}

\subsection{Norm Resolvent Convergence when $B$ is bounded}\ \\

Let us now consider the setting where  $B$ being bounded ensures  $A + \beta B$ is closed.
\subsubsection{Convergence if hermitianized anticommutators are bounded from below}\ \\
We begin our discussion without any assumptions on how 
$A$ interacts with kernel and range of $B$. Instead we note that since $\mathcal{B}(\mathcal{H})$ is Banach, a sufficient condition for generalized resolvent convergence is that the family $\{(A + \beta B)^{-1}\}_\beta$ forms a Cauchy net. As we show now, this can be ensured if the 'hermitianized' anti-commutators $(AB^* + BA^*)$ and $(BA^* + AB^*)$ are bounded from below:

\begin{Thm}\label{main_closed_theorem}
	Let \( A: \mathcal{D}(A) \rightarrow \mathcal{H} \) be a closed operator. Let \( B \in \mathcal{B}(\mathcal{H}) \) be bounded. Assume further that $0\in \sigma(B)$ is an isolated point of $\sigma(B)$ at which the quasi-nilpotent part of $B$ vanishes. Assume the Riesz projection $P_0$  (associated to $B$) onto $0\in \sigma(B)$ is an orthogonal projection.  Let $Q_0 = Id_{\mathcal{H}} - P_0$ be the corresponding complementary projection. Assume  $\exists \gamma \geq 0$ so that for all $ \phi \in \mathcal{D}(A)$ and $\psi \in \mathcal{D}(A^*)$ we have
	\begin{align}\label{forms_bounded+from_below}
		-\gamma \|Q\phi\|^2 \leq \langle A\phi, B\phi \rangle + \langle B\phi, A\phi \rangle,  \ \
		-\gamma \|Q\psi\|^2 \leq \langle A^*\psi, B^*\psi \rangle + \langle B^*\psi, A^*\psi \rangle.
	\end{align}
	Then there exists a closed  operator \( \underline{A} : \mathcal{D}(\underline{A}) \to P\mathcal{H} \) such that
	\[
	\left\|(A + \beta B - z)^{-1} - P (\underline{A} - z)^{-1} P \right\| = \mathcal{O}\left(1/\beta \right).
	\]
	If $\mathcal{D}(A) \cap P \mathcal{H}$ is dense in
	$P\mathcal{H}$, then $\underline{A}$ is densely defined on $P\mathcal{H}$ and a closed extension of $PA: \mathcal{D}(A)\cap P\mathcal{H} \rightarrow P \mathcal{H} $.
\end{Thm}
Before conducting the proof, let us first separate out a discussion of the implications of the conditions in Theorem \ref{main_closed_theorem}. We begin with discussing the implications of assuming that the Riesz projector is orthogonal:

\begin{Lemma}\label{Riesz_proj_proof}
	If the Riesz projector $P$  is orthogonal, there is a constant $C > 0$ such that   $|(\eta,  B \eta)| \leq C \|B\eta\|^2$ for all $ \eta \in \mathcal{H}$.
\end{Lemma}

\begin{proof}
	Clearly if $P$ is orthogonal, also $Q = I - P$ is orthognal. 
	Thus we have $|(\eta,  B \eta)| = |(Q\eta,  B Q\eta)| \leq \|B\| \|Q \eta\|^2$. It is hence sufficient to establishe the existence of a constant $\tilde{C}$ so that $\|Q \eta\|^2 \leq \tilde{C} \|B \eta\|^2$. Since   $0 \notin \sigma(Q BB Q)$ (see e.g. \cite{HislopSigal1996}),  $Q BB Q: Q \mathcal{H} \rightarrow  Q \mathcal{H}$ is boundedly invertible. For any $y \in \mathcal{H}$ we thus have
	$
	|\langle [Q BB Q]^{-1} Qy ,[Q BB Q]^{-1} Qy  \rangle| \leq 
	\|[Q BB Q]^{-1}\|^2 \cdot \|Q y\|.
	$
	Setting $\eta = By$ shows we may set $C = \|B\|^2 \cdot \|[Q BB Q]^{-1}\| $.
\end{proof}

Next we discuss the implications of the assumption on lower boundedness in (\ref{forms_bounded+from_below}):

\begin{Lem}\label{bddness_lemma}
	The lower-boundedness of the forms in (\ref{forms_bounded+from_below}) together with the assumption $P = P^* $ implies the exitence of
	a constant 
	$C_z \geq 0$ so that $\|Q( B + \delta (A-z))^{-1}\|$, $\|(B + \delta (A-z))^{-1}Q\| \leq C_z$ uniformly as $\delta \to 0$.
\end{Lem}
\begin{proof}
	We begin by establishing the existence of $C_z$ so that  $\|Q( B + \delta (A-z))^{-1}\| \leq C_z$. This condition is equivalent to 
	$
	\left\| Q (B + \tilde{\delta} (A - z))^{-1} \psi \right\|^2 \leq \tilde{C} \|\psi\|^2$,  $\forall \psi \in \mathcal{H}
	$.
	This in turn is true  if
	$
	\|Q \eta\|^2 \leq \tilde{C} \| B \eta + \tilde{\delta} (A - z) \eta \|^2$,  $\forall \eta \in \mathcal{D}(A).
	$
	Expanding  yields
	\[
	\|Q \eta\|^2 \leq \tilde{C} \left[ \tilde{\delta}^2 \|(A - z) \eta\|^2 + \|B \eta\|^2 + \tilde{\delta} \left( \langle (A - z) \eta, B \eta \rangle + \langle B \eta, (A - z) \eta \rangle \right) \right].
	\]
	Since $\tilde{\delta}^2 \|(A - z) \eta\|^2 $ is non-negative, the above inequality is certainly true, if we know that the following tightened inequality holds:
	\[
	\|Q \eta\|^2 \leq \tilde{C} \left[  \|B \eta\|^2 + \delta \left( \langle A  \eta, B \eta \rangle + \langle B \eta, A  \eta \rangle \right)
	-
	\delta \left( \langle z  \eta, B \eta \rangle + \langle B \eta, z  \eta \rangle \right) 
	\right].
	\]
	
	This -- in turn -- is true if the following three inequalities hold uniformly as $\delta \to 0$:
	\begin{subequations}\label{ineq:main}
		\begin{align}
			\|Q \eta\|^2 &\leq \frac{C_z}{3} \|B\eta\|^2 \label{ineq:a} \\
			0 &\leq  \frac{\|B\eta\|^2}{3} + \delta \left( \langle A  \eta, B \eta \rangle + \langle B \eta, A  \eta \rangle \right) \label{ineq:b} \\
			0 &\leq  \frac{\|B\eta\|^2}{3}  -
			\delta \left( \langle z  \eta, B \eta \rangle + \langle B \eta, z  \eta \rangle \right) \label{ineq:c}
		\end{align}
	\end{subequations} 
	By the proof of Lemma \ref{Riesz_proj_proof}, we know that (\ref{ineq:a}) holds if $C_z \geq 3 \|[Q B Q]^{-1}\|^2$. Again by Lemma \ref{Riesz_proj_proof}, we know that (\ref{ineq:c}) is true for $C_z \geq  3 |z|\|B\|^2 \|[Q B Q]^{-1}\|$. Finally, we note that with our assumption 
	\begin{align}
		&\frac{\|B\eta\|^2}{3} + \delta \left( \langle A  \eta, B \eta \rangle + \langle B \eta, A  \eta \rangle \right)
		\geq   \frac{\|B\eta\|^2}{3}  - \delta \gamma \|Q \eta\| ^2\\
		\geq&  \left( \frac{1}{3}- \delta \gamma   \|[Q B Q]^{-1}\|^2\right)\|B\eta\| ^2,
	\end{align}
	which is eventually non-negative uniformly in $\eta \in \mathcal{H}$ and $\delta \to 0$, so that (\ref{ineq:b}) holds as well. Hence we may choose $C_z = 3\cdot  \max\left\{1, |z|\|B\|^2\right\}\cdot  \|[Q B Q]^{-1}\|^2$.
	
	To establish the existence of $C_z \geq 0$ so that  $\|(B + \delta (A-z))^{-1}Q\| \leq C_z$ uniformly, we note  $\|(B + \delta (A-z))^{-1}Q\| = \|[Q]^* \cdot [(B + \delta (A-z))^{-1}]^*\| = \|Q\cdot [(B^* + \delta (A^*-\overline{z}))^{-1}]| $.
	From here we may proceed in complete analogy to above, using the remaining condition in (\ref{forms_bounded+from_below}).
	
\end{proof}

Equipped with the above two Lemmata, we are now ready prove Theorem  \ref{main_closed_theorem}:
\begin{proof}[Proof of Theorem \ref{main_closed_theorem}]
	We begin by establishing that \(\left\{ (A + \beta B - z)^{-1} \right\}_{\beta \geq 0}\) is a Cauchy net. Indeed, by the second resolvent formula, we have
	\begin{align}
		&\left\| (A + \beta B - z)^{-1} - (A + \tilde{\beta} B - z)^{-1} \right\| = |\beta - \tilde{\beta}| \cdot \left\| (A + \beta B - z)^{-1} B (A + \tilde{\beta} B - z)^{-1} \right\|\\
		=& \frac{|\beta - \tilde{\beta}|}{\beta \tilde{\beta}} \cdot \left\| (B + \delta (A - z))^{-1} Q B Q (B + \tilde{\delta} (A - z))^{-1} \right\|,
	\end{align}
	where we have set \(\delta = \beta^{-1}\), \(\tilde{\delta} = \tilde{\beta}^{-1}\). By Lemma \ref{bddness_lemma}, there is a $C \geq 0$ so that  
	\begin{align}
		&\left\| (B + \delta (A - z))^{-1} Q B Q (B + \tilde{\delta} (A - z))^{-1} \right\|\\
		\leq &\left\| (B + \delta (A - z))^{-1} Q \right\| \cdot \|B\| \cdot \left\| Q (B + \tilde{\delta} (A - z))^{-1} \right\| \leq C^2 \|B\|.
	\end{align}
	
	Hence the net \(\left\{ (A + \beta B - z)^{-1}\right\}_\beta\) is indeed Cauchy, and
	since \(\mathcal{B}(\mathcal{H})\) is Banach, there thus exists \(T_z \in \mathcal{B}(\mathcal{H})\) such that
	$
	\left\| (A + \beta B - z)^{-1} - T_z \right\| \to 0.
	$.
	W.l.o.g. take now \(\tilde{\beta} \geq \beta\). Then
	$
	\| (A + \beta B - z)^{-1} - (A + \tilde{\beta} B - z)^{-1} \| \leq C^2 \|B\|/\beta.
	$
	Next we note
	\begin{align}
		\left\| (A + \beta B - z)^{-1} - T_z \right\| 
		\leq& \left\| (A + \beta B - z)^{-1} - (A + \tilde{\beta} B - z)^{-1} \right\| \\
		+& \left\| (A + \tilde{\beta} B - z)^{-1} - T_z \right\|\\
		\leq& 2 \tilde{C} \|B\|/\beta,
	\end{align}
	with the last inequality holding because the second term on the previous line can be made arbitrarily small (as \(\tilde{\beta} \to \infty\)).
	Next we note that the mapping \(z \mapsto T_z\) inherits the first resolvent formula. The first resolvent formula for \((A + \beta B - z)^{-1}\) yields
	\[
	(A + \beta B - z)^{-1} - (A + \beta B - y)^{-1} = (z - y) (A + \beta B - z)^{-1} (A + \beta B - y)^{-1}.
	\]
	
	The claim then follows since the left side converges in norm to \(T_z - T_y\), while the right side converges to \((z - y) T_z T_y\).
	Indeed, for any two nets of bounded operators \(A_\beta \to A\), 
	\(B_\beta \to B\),
	$
	\| A_\beta B_\beta - A B \| \leq \|A\| \| B_\beta - B \| + \| B_\beta \| \| A_\beta - A \|.
	$
	
	Next we establish, that \(T_z : P \mathcal{H} \to P \mathcal{H}\) is injective: Suppose \(T_z P \psi = 0\). Set \(\eta_\beta := (A + \beta B - z)^{-1} P \psi\).  Then \(T_z P \psi = \lim_{\beta \to \infty} \eta_\beta\).

	Clearly, \( P(A - z)\eta_\beta = P(A + \beta B - z) \eta_\beta = P \psi \). Since \( PA \eta_\beta = P \psi + z \eta_\beta \) and \( \eta_\beta \rightarrow 0 \), the sequence \( \{PA \eta_\beta\}_\beta \) converges to \( P \psi \). However, since \( \eta_\beta \rightarrow 0 \) and \( PA \) is closed, we have for the limit \( P \psi \) of \( PA \eta_\beta \) that
	$
	P \psi = 0.
	$
	This establishes injectivity of $T_z$.

	By \cite[Chapter VIII.4, Theorem 1]{yosida1968functional}, since \(z \mapsto T_z\) is a pseudo-resolvent with dense range  on the Hilbert space $\overline{\text{Ran}(T_z)}$ 
	and \(\ker(T_z) = \{0\}\), \(T_z\) is generated by a closed, densely defined operator \(\underline{A}\) on $\overline{\text{Ran}(T_z)}$.
	
	Suppose now \(\mathcal{D}(A) \cap P \mathcal{H}\) is dense in \(P \mathcal{H}\). For \(\phi \in \mathcal{D}(A) \cap P \mathcal{H}\),
	\[
	(A + \beta B - z)^{-1} (A - z) \phi = (A + \beta B - z)^{-1} (A + \beta B - z) \phi = \phi,
	\]
	so \(\phi \in \operatorname{Ran}(T_z) = \mathcal{D}(\underline{A})\). Hence \(\underline{A}\) is densely defined ($\overline{\mathcal{D}(\underline{A})} \equiv \overline{ \operatorname{Ran}(T_z)} = P\mathcal{H}$).
\end{proof}

From the proof may also understand the necessity of the condition that the Riesz projector $P$ be orthogonal: By Lemma \ref{Riesz_proj_proof} it ensures that  we may bound the quadratic form associated to $B$ by the one associated to $B^*B$, which via  (\ref{ineq:c}) then allows us to conclude that the family $\{(A + \beta B)^{-1}\}_\beta$ is Cauchy. We may drop the assumption that $P^* = P$, if instead we make further assumptions on how the generically unbounded operator $A$ behaves on $\ker(B)$ and $\text{Ran}(B)$, as we discuss in the subsequent Section \ref{nice_interaction}. Here we however first provide an example illustrating Theorem \ref{main_closed_theorem}:

\begin{Ex}\label{high_energy_example_II} Let us revisit the Dirac Hamiltonian setting of Example \ref{weak_interaction_example}, now with gauge background completely switched off. The free Hamiltonian is then given as 
\begin{align}
	A = \begin{pmatrix}
	- i \partial_x  \mathbb{I}_2 & m \mathbb{I}_2 \\
	m \mathbb{I}_2 & i \partial_x  \mathbb{I}_2
\end{pmatrix}.
\end{align}	
It acts on $\Psi = (\psi_L,\psi_R)$, with $
\psi_L := (\psi_L^{(1)} \ \psi_L^{(2)} )^\top $ and $
\psi_R :=( \psi_R^{(1)} \ \psi_R^{(2)})^\top$.
Changing our representation $\Psi \rightarrow \Phi$ from making hermiticity apparent ($\Psi$) to making Fermion-type apparent ($\Phi$), with $\Phi = (\psi^{(1)}, \psi^{(2)})$ and $\phi^{(a)} = (\psi^{(a)}_L, \psi^{(a)}_R)$, our base Hamiltonian may equivalently be represented as 
\begin{align}
	A = \begin{pmatrix}
		- i \sigma_x \partial_x  + m \sigma_z & 0 \\
		0 &- i \sigma_x \partial_x + m \sigma_z
	\end{pmatrix}.
\end{align}	
Consider now a perturbation of the form $k^{(1)}_y \cdot B$, with
\begin{align}
	B = \begin{pmatrix}
		 \sigma_y & 0 \\
		0 &0
	\end{pmatrix}.
\end{align}	
The corresponding operators $A_{k^{(1)}_y} = A + k^{(1)}_y B$ can be thought of as the effective one-dimensional Hamiltonian arising from an original two-dimensional Dirac Hamiltonian at fixed momenta in $y$-direction (c.f. Appendix \ref{fixed_momentum_de}): For the first Fermion type $\psi^{(1)}$ we consider the momentum $p_y$ in $y$-direction to be given as $k^{(1)}_y$. For the second fermion type $\psi^{(2)}$, we fix $p_y = 0$. \emph{Strong} resolvent convergence as the momentum $k^{(1)}_y $ is taken to infinity is then already guaranteed by Theorem \ref{convergence_to_compression} of Section \ref{self_adjoint_section}. Theorem \ref{main_closed_theorem} above now promotes this to \emph{norm}-resolvent convergence of the Hamiltonian $A + k^{(1)}_y B$ towards the effective Hamiltonian $\underline{A} = - i \sigma_x \partial_x + m \sigma_z$ acting purely on $\psi^{(2)}$ as $\beta = k^{(1)}_y \rightarrow \infty$. Indeed, since $B$ is an orthogonal projection in and of itself, we only need to check the lower-boundedness of the forms in (\ref{forms_bounded+from_below}). However, since Pauli matrices are Hermitian and anti-commute (and thus $[\sigma_y \sigma_i + \sigma_i \sigma_y] = 0$ for $i = x,z$), these forms vanish identically and are thus trivially bounded from below. Hence Theorem \ref{main_closed_theorem} applies and establishes convergence in norm.
\end{Ex}

\subsubsection{Convergence if $A$ is bounded outside of  $\ker(B)$}\label{nice_interaction}\ \\

\noindent In the previous subsection, we had to restrict ourselves to the setting where the Riesz projector is orthogonal, and the quadratic forms of (\ref{forms_bounded+from_below}) are bounded from below.
If we know more about the behaviour of $A$ on range and kernel of  $B$, stronger results can be established. Here we will assume that the compressed operator $PA: \mathcal{D}(A) \cap P\mathcal{H} \rightarrow P\mathcal{H}$ is denely defined on $P\mathcal{H}$ as well as closed. The closedness of $A$ implies that $\mathcal{D}(A)$ equipped with the graph norm is a Hilbert space in which $\mathcal{D} := \mathcal{D}(A) \cap P\mathcal{H} $ is a closed subspace. We denote the orthogonal complement of $\mathcal{D}$ by $\mathcal{D}^\intercal$, so that $ \mathcal{D}(A) = \mathcal{D} \oplus   \mathcal{D}^\intercal$. Below we will also make use of the closure  of this orthogonal subspace in the original Hilbert space ($\overline{\mathcal{D}^\intercal}^{\|\cdot\|_{\mathcal{H}}}$).
With these preparations we are now ready to state our result:

\begin{Thm}\label{thm_bounded_on_range} 
	Let $B \in \mathcal{B}(\mathcal{H})$ have zero as an isolated eigenvalue at which the corresponding nilpotent part vanishes. Let $P$ be the Riesz projector associated to $0 \in \sigma(B)$ and let $Q = Id - P$ be its complementary projection. 	
	Let $A: \mathcal{D}(A) \rightarrow \mathcal{H}$ be a closed operator whose compression $PA: \mathcal{D} \rightarrow P\mathcal{H}$ is also closed, so that there exists  $z \in \rho(A) \cap \rho(PA)$. Futher assume that $QA$ and $A\restriction_{\mathcal{D}^\intercal}$ extend to bounded operators on the closure of their domains. Finally assume $Q: \overline{\mathcal{D}^\intercal}^{\|\cdot\|_{\mathcal{H}}} \rightarrow Q\mathcal{H}$ has a bounded inverse.
	Then for sufficiently large $\beta \gg 1$ the operator $(A + \beta B - z)$ is boundedly invertible and we have the generalized norm resolvent convergence
	\begin{align}\label{convergence_preserved_domain}
		\|(A + \beta B - z)^{-1} - P(PA -z)^{-1}P\| \lesssim 1/\beta \xrightarrow{\beta \rightarrow \infty} 0.
	\end{align}
 \end{Thm}

\begin{proof}
	We note that we may view  $A_\beta = A + \beta B -z$ as a bounded operator between the Hilbert spaces $D \oplus D^\intercal \rightarrow P\mathcal{H} +Q \mathcal{H}$, with the latter sum not necessarily being orthogonal. This induced a block decomposition of $A_\beta -z$ as 
	\begin{align}\label{block_operator_I}
	(A_\beta  - z)
		=
		\begin{pmatrix}
			[PA - zP]\restriction_{\mathcal{D}} & [PA - zP]\restriction_{\mathcal{D}^\intercal}   \\
			QA\restriction_{\mathcal{D}}    & [QA + \beta QB - zQ]\restriction_{\mathcal{D}^\intercal}
		\end{pmatrix}.
	\end{align}
	By assumption, the off-diagonal entries constitute bounded operators. 
We now use a Schur complement approach to show that the operator in (\ref{block_operator_I}) is boundedly invertible. To build  a candidate inverse $T$ let first define 	$ 	S := [QA + \beta QB - zQ - QAP R PA]\restriction_{\mathcal{D}^\intercal}  \in \mathcal{B}(\overline{\mathcal{D}^\intercal}^{\| \cdot \|_{\mathcal{H}}}, Q  \mathcal{H}),
$
and 	$
R := (PA\restriction_{\mathcal{D}} - zP)^{-1} \in \mathcal{B}(P  \mathcal{H}, \mathcal{D})$. 
 Since $z \notin \sigma(PAP)$ by assumption, we know that $R$ is well-defined. We then define the candidate inverse as:
\begin{align}\label{def_block_inverse}
	T =
	\begin{bmatrix}
		R + R P[A\restriction_{\mathcal{D}^{\intercal}}] S^{-1} Q[A\restriction_{\mathcal{D}}]P R & -R P[A\restriction_{\mathcal{D}^{\intercal}}]S^{-1} \\
		- S^{-1} Q[A\restriction_{\mathcal{D}}] R & S^{-1}
	\end{bmatrix}.
\end{align}
To show that $S: \overline{\mathcal{D}^\intercal}^{\|\cdot\|_{\mathcal{H}}} \rightarrow Q\mathcal{H} $  has an inverse $S^{-1} \in \mathcal{B}(Q \mathcal{H}, \overline{\mathcal{D}^\intercal}^{\|\cdot\|_{\mathcal{H}}})$, we use the assumption that $Q: \overline{\mathcal{D}^\intercal}^{\|\cdot\|_{\mathcal{H}}} \rightarrow Q\mathcal{H}$ is boundedly invertible. Denoting this inverse by $Q^- \in \mathcal{B}(Q \mathcal{H}, \overline{\mathcal{D}^\intercal}^{\|\cdot\|_{\mathcal{H}}})$, we have 
\begin{align}
	S = \beta QBQ \cdot Q \cdot  \left[Id_{\overline{\mathcal{D}^\intercal}^{\|\cdot\|_{\mathcal{H}}} }+\frac{1}{\beta}Q^{-}(QBQ)^{-1}   
	(QA
	- zQ - Q[A\restriction_{\mathcal{D}}] R PA)
	)       
	 \right].
\end{align}
We note that $(QBQ)^{-1} \in \mathcal{B}(Q\mathcal{H})$ and the term 
$(QA
- zQ - Q[A\restriction_{\mathcal{D}}] R PA)
)    $ extends to a bounded operator defined on all of $\overline{\mathcal{D}^\intercal}^{\|\cdot\|_{\mathcal{H}}}$. Thus for sufficiently large $\beta \gg 1$ the term in square brackets is invertible by a Neumann series as
\begin{align}\label{neumann_expression_s}
	S^{-1} = \frac{1}{\beta}  \left[\sum_{k = 0}^{\infty} \left[-Q^{-}(QBQ)^{-1}   
	(QA
	- zQ - Q[A\restriction_{\mathcal{D}}] R PA)\right]^k/\beta^k\right]
	\cdot Q^{-} \cdot
	[QBQ]^{-1} .
\end{align}
Verifying that $T(A + \beta B -z) = (A + \beta B -z)T = Id_{\mathcal{H}}$ boils done to block matrix multiplication; included in Appendix \ref{proof_support} for completeness.
From (\ref{neumann_expression_s}) we also infer $\|S^{-1}\| \lesssim 1/\beta \rightarrow 0$. The convergence in (\ref{convergence_preserved_domain}) is then established by noting that due to this fact, all but the top-left block in  (\ref{def_block_inverse}) are $\mathcal{O}(1/\beta)$.
\end{proof}

\begin{Ex}\label{compressed_example_II}
As an example, let us consider a closed, densely defined,  maximally dissipative  operator $A$. We decompose the original Hilbert space into an orthogonal sum as $\mathcal{H} = \mathcal{F} \oplus \mathcal{G}$, with $\dim \mathcal{G} < \infty$. Then with $P$ the orthogonal projection onto $\mathcal{F}$, the compression of $A$ defined as $PA: \mathcal{D} \rightarrow P\mathcal{H}$ (with $\mathcal{D} = \mathcal{D}(A) \cap P\mathcal{H}$) is densely defined and closed (as well as maximally dissipative itself) \cite[Theorem 2.2]{Nudelman2011}.  Furthermore
the domain of $A$ decomposes as the direct sum
$\mathcal{D}(A) = \mathcal{D} \oplus \mathcal{D}^\intercal$, with $\mathcal{D} = \mathcal{D}(A) \cap \mathcal{F}$ and $\dim \mathcal{G} = \dim\mathcal{D}^\intercal < \infty$ \cite{Nudelman2011}.

Let now $B$ be any operator with $ \text{ran}(B) = \mathcal{G}$ and  $\ker(B) = \mathcal{F}$  so that the Riesz projector at $0 \in \sigma(B)$ is orthogonal. 
 Then $A + \beta B $ converges in the generalized norm resolvent sense towars the compression of $A$ for any $z$ in the lower hemisphere of $\mathds{C}$. Indeed: Since $A$ and $PA$ are both dissipative, the entire lower hemisphere is contained in their respective resolvent sets. Since $\mathcal{G}, \mathcal{D}^\intercal$ are finite dimensional, the ranges of the operators $QA, A\restriction_{\mathcal{D}^\intercal}$ are finite dimensional, so that both operators triviallly extend to bounded operators on the closure of their inital domains. It remains to check that $Q:  \overline{\mathcal{D}^\intercal}^{\|\cdot\|_{\mathcal{H}}} \rightarrow Q \mathcal{H}$ is boundedly invertible. But $D^\intercal$ is finite dimensional and agrees with its closure. Moreover, $Q\restriction{D^\intercal} $ is clearly injective, so that the claim follows since $Q\restriction{D^\intercal} $ is an injective map between finite dimensional spaces of the same dimension.
\end{Ex}

Next we note that if $P$  in fact preserves the domain of $A$, we can relax the other assumptions in Theorem \ref{thm_bounded_on_range}:

\begin{Cor}\label{main_cor}
	 Assume $P \mathcal{D}(A) \subseteq \mathcal{D}(A)$ and that $P \mathcal{D}(A)$ is dense in $P\mathcal{H}$. Further assume the operators $QA, AQ$ (defined on their natural domains) extend to bounded operators $QA, AQ: \mathcal{H} \rightarrow \mathcal{H}$. Further assume there is $z \in \mathds{C}$ so that $z \in \rho(PAP)$.
	Then for sufficiently large $\beta \gg 1$ the operator $(A + \beta B - z)$ is boundedly invertible and we have the generalized norm resolvent convergence
		\begin{align}\label{convergence_preserved_domain_II}
		\|(A + \beta B - z)^{-1} - P(PAP -z)^{-1}P\| \lesssim 1/\beta \xrightarrow{\beta \rightarrow \infty} 0.
	\end{align}
\end{Cor}

\begin{proof}
	We begin with a discussion of $PAP: [P D(A)] \rightarrow P\mathcal{H}$: Since  $P D(A)$ is dense in $PH$ this operator is densely defined. Furthermore note that since  $AP: [P D(A)] \rightarrow \mathcal{H}$ is closed, also $PAP$ is closed. Indeed, since $P$ is bounded we only need to establish closedness of $AP$. Hence assume there is a sequence $\{x_i\}_i$ in $P\mathcal{D}(A)$ so that $x_i \rightarrow x$ and $APx_i \rightarrow y$. Then  $\{x_i\}_i$ is also a sequence in $\mathcal{D}(A)$ so that $Ax_i \rightarrow y$. The result then follows from closedness of $A$ on its original domain.
	
	Next we set $AQ = A - AP$ on $\mathcal{D}(A)$ and by assumption are guaranteed that this extends to a bounded operator on all of $\mathcal{H}$. We then note 
	$
	A + \beta B -z = PAP - zP + PAQ + QAP + QAQ + QBQ -zQ$,
	so that our operator family has block structure as in the proof of Theorem \ref{thm_bounded_on_range}, only with $\overline{\mathcal{D}^{\intercal}}^{\|\cdot\|_{\mathcal{H}}}$ replaced by $Q\mathcal{H}$. To finish the proof of the Corollary, we can then use the same block-inverse argument as above.
\end{proof}

\begin{Ex}\label{main_graph_example}

As an example of the usefulness of Corollary \ref{main_cor}, we extend finite dimensional results in \cite{koke2024holonets} to the infinite dimensional setting:
Let us consider the Laplacian associated to a countably infinite  locally finite directed weighted graph $G = (V, m, E, b)$. Here $V$ is a countably infinite set of vertices and $m: V \rightarrow (0, \infty)$ is a weight function assigning a positive mass $m(v)$ to each node $v \in V$. The set 
 $E \subseteq V \times V$ consists of the edges in the graph $G$ and $a: V \times V \rightarrow [0,\infty)$ is a weight function so that $a(v,w) > 0$ if and only if $(v,w) \in E$.
Crucially, since we consider directed graphs, $a$ is not necessarily symmetric in its arguments (generically $a(v,w) \neq a(w,v)$).
 On the weighted $\ell^2$-space
\begin{align}
	\ell^2(G,m) = \left\{ f: V \rightarrow \mathds{C}: \sum\limits_{v \in V} m(v) |f(v)|^2 < \infty      \right\}
\end{align}
with inner product $\langle f,g\rangle_{\ell^2(V)} = \sum_{x \in V} \overline{f(x)} g(x) $, we can then introduce the graph Laplacian  via its local action as
 \begin{align}\label{initial_L_action}
 	[L_af](x) = \frac{1}{m(x)} \sum\limits_{y \in V} a(x,y) (f(x) - f(y)),
 \end{align}
initially defined as an operator with domain the functions $f$ of compact support.

 To limit the technical complexity introduced by working with directed graphs, we will now make two assumptions: First, we assume that at each node the in-degree equals the out-degree
\begin{align}
	d^{\text{in}}_{x} :=	\sum\limits_{y \in G} a(x,y) = \sum\limits_{y \in G} a(y,x) =:   d^{\text{out}}_{x}.
\end{align} 
This is known as a  'Kirchhoff assumption' in the literature, since it resembles Kirchhoff's first law in electrical circuits \cite{anne2018sectoriality, balti2017,balti2017a, balti2017b}.
Second, we make the technical assumption that there exists $M \geq 0$ so that $\sum_{w \in V} |a(v,w) - a(w,v)| \leq M m(v) $ for each node $v \in V$. Under these two assumptions, the operator $L_a$ on its initial domain is sectorial \cite{anne2018sectoriality}, so that by \cite[Theorem V.3.4]{kato1995perturbation} $L_a$ it is closable. Below, we will hence consider the operator $(L_a, D(L_a))$ arising as the closure of the initial operator defined in (\ref{initial_L_action}).

We now pick out a  finite set of vertices $W \subseteq V$ and introduce a new weight function $b$ that only considers weights between nodes in $W$ as
$
b(x,y) = \chi_W(x) \chi_w(y)	b(x,y)
$,
with $\chi_W$ the characteristic function of the set vertex subset $W$. We will here assume that also the subgraph determined by $W$ is connected, in the sense that there is at least one node from which all other nodes in $W$ may be reached via a path along edges $(x,y)$ for which none of the weights $b(x,y)$ vanishes.  We also assume that the subgraph determined by $W$ satisfies the Kirchhoff assumption (i.e. $\sum_{y \in G} b(x,y) = \sum_{y \in G} b(y,x)$).\footnote{Treating the general case outside any such assumptions is completely within reach using the tools developed in this section.
	Since however the corresponding graph theoretic interpretations become significantly more involved, we investigate this general setting 
in a companion paper \cite{koke_graph26}.}

 In analogy to (\ref{initial_L_action}), we can then introduce the  Laplacian corresponding to the subgraph $W$ as
 \begin{align}\label{initial_L_action_finite}
	[L_bf](x) = \frac{1}{m(y)} \sum\limits_{x \in V} b(x,y) (f(x) - f(y)),
\end{align}
which after extending by zero outside of $W$ clearly defines a bounded operator on $\ell^2(V)$. 
Since the range of $L_b$ is finite dimensional, its spectrum is discrete. By standard results in graph theory, the algebraic multiplicity of the eigenvalue $0 \in L_b$ is equal to its geometric multiplicity \cite{CaughmanVeerman2006}, so that the nilpotent part of the corresponding Riesz projector $P$ vanishes. In fact, since we assumed that there is a node from which all other nodes in $W$ can be reached via directed paths completely contained in this subgraph, $P$ is even an orthogonal projection, with its action given as $[Pf](x) = \chi_{G \setminus W}(x)f(x) +     \langle \chi_W,f \rangle_{\ell^2(V)} \chi_W(x)$ \cite{Veerman2020APO}. We now consider the setting $A + \beta B = L_a + \beta L_b$ where we scale up the weights in the subgraph $W$. 
We are then interested in the effective Laplacian $\underline{L_a}$ describing the graph in the infinite weight limit.

We first note that $P\mathcal{H}$ is precisely comprised of those functions $f \in \ell^2(V)$ that are constant on $W$. We nay thus defined a new node set $\underline{V} = V\setminus W \cup \{W\}$ obtained by removing all nodes of $W$ from $V$ and adding back only a single node   $w := \{W\}$ representing the cluster $W$. We may thus unitarily identify $P \mathcal{H}$ with the Hilbert space $\ell^2(\underline{V})$, where the inner product between two functions $\underline{f}, \underline{g}: \underline{W} \rightarrow \mathds{C}$ in the image of $P$ is determined as 
\begin{align}
\underline{m}(w) \overline{\underline{f}(w)} \underline{g}(w) +	 \sum\limits_{x \in V\setminus W} m(x) \overline{f(x)} g(x) =: \langle \underline{f}, \underline{g} \rangle_{\ell^2(W)} \equiv \langle Pf, Pg \rangle_{\ell^2(V)},
\end{align}
with $f,g \in \ell^2(V)$ arbitrary functions in the pre-images $P^{-1}(f), P^{-1}(g)$. Thus we have $\underline{m}(w) = \sum_{x \in W} m(x)$ and $\underline{f}(w) = \frac{1}{\underline{m}} \sum_{x \in W} f(x)$.

We will now show that the limit effective operator guaranteed to exist by Corollary  \ref{main_cor} is precisely given as the closure of the graph Laplacian 

 \begin{align}\label{initial_L_action_reduced}
	[\underline{L_a}\underline{f}](x) = \frac{1}{\underline{m}(x)} \sum\limits_{y \in \underline{V}} \underline{a}(x,y) (\underline{f}(x) - \underline{f}(y)),
\end{align}
where $\underline{a}(w,y) = \sum_{x \in  W } a(x,y)$ and similarly $\underline{a}(y,w)$ is also given via an aggregation of the connections into the cluster $W$. To apply Theorem \ref{thm_bounded_on_range}, we check all its conditions are met: 
 Since $P$ acting on $f$ only linearly modifies $f$ locally within $W$ and any function completely supported in $W$ is in the domain of $L_a$, we have $P \mathcal{D}(L_a) \subseteq \mathcal{D}(L_a)$. Using our identification $\ell^2(W) \equiv P\mathcal{H}$, it is clear that $P\mathcal{D}(L_a)$ in particular contains all functions that are compactly supported on $\underline{V}$, so that $P\mathcal{D}(L_a)$ is dense in $\ell^2(\underline{V})$.
Next we note that the complementary projection $Q= Id_{\ell^2(V)} - P$ is given simply as $[Qf](x) = \chi_W(x)f(x) - \langle \chi_W,f \rangle_{\ell^2(V)} \chi_W(x) $. Clearly the range of $Q$ is finite dimensional, so that $QA,AQ$ are well defined and extend to bounded operators on all of $\ell^2(V)$. Since the numerical range of  $L_a$ is contained in the right hemisphere of $\mathds{C}$ \cite[Proposition 2.2]{anne2018sectoriality}, we have $\lambda \in \rho(L_a)$ for any $\lambda \in (-\infty, 0)$ \cite{ArlinskiZagrebnov2010,kato1995perturbation}. Since $P$ is an orthogonal projection, the numerical range may only contract under $L_a \mapsto P L_a P$. Thus $\lambda$ is also not in the spectrum of $\underline{L_b}$. Hence the conditions of Theorem \ref{thm_bounded_on_range} are met and we have generalized norm resolvent convergence of $L_a + \beta L_b$ towards $PL_aP$. It is then a straightforward exercise to establish that under the identification $P\mathcal{H} \cong \ell^2(\underline{V})$, the operator $PL_aP$ exactly corresponds to the closure of the graph Laplacian  in (\ref{initial_L_action_reduced}).

\end{Ex} 
	\appendix

	\section{Chiral Dirac Equation in an \(\mathrm{SU}(2)\) Gauge Background}
	
	In this appendix, we provide further context and details for Example \ref{weak_interaction_example}. The derivations below are standard, as general references a number of books on particle physics or the Dirac equation may serve \cite{peskin1995introduction, hamilton2017mathematical, Bettini2014, srednicki2007quantum, Thaller1992}.
	
	\medskip
	
	In the Standard Model of particle physics, the weak interaction is mediated by an \(\mathrm{SU}(2)\) gauge field acting on left-handed fermion doublets. This chiral gauge interaction is responsible for processes such as beta decay and neutrino scattering. The \(\mathrm{SU}(2)\) symmetry acts only on the left-chiral components of fermions, organized into doublets (e.g., \((\nu_e, e)_L\)), while the right-chiral components transform trivially under the gauge group. In this section, we isolate this weak-sector structure in a classical gauge background and derive the associated Dirac equation in an explicit, componentwise form suitable for mathematical analysis.

\subsection{1 + 1 Dimensions}\label{1+1_d_weak_interaction_derivation}
	
Mathematically speaking, we here hence consider a system of spinor fields on Minkowski space \(\mathbb{R}^{1+1}\), interacting with a fixed, classical \(\mathrm{SU}(2)\) gauge field. The unknown is a field \(\Psi: \mathbb{R}^{1+1} \to \mathbb{C}^4\), which decomposes as a doublet of Dirac spinors:
\[
\Psi = \begin{pmatrix}
	\psi^{(1)} \
	\psi^{(2)}
\end{pmatrix}^\top, \quad
\psi^{(a)} = \begin{pmatrix}
	\psi_L^{(a)} \
	\psi_R^{(a)}
\end{pmatrix}^\top, \quad a = 1, 2,
\]
where each \(\psi_L^{(a)}\) and \(\psi_R^{(a)}\) is a complex-valued function of spacetime. These represent the two so called 'chiral' components (\textbf{L}eft and \textbf{R}ight) of a two-component Dirac spinor in \(1+1\) dimensions \cite{hamilton2017mathematical}.  
The gauge field $W $ consists of functions $W_{\mu}$ ($\mu = 0,1$) where each function $W_\mu$ maps from spacetime $\mathds{R}^{1+1}$ into the Lie algebra \(\mathfrak{su}(2)\). This is a three dimensional Lie algebra, whose generators $\{\tau^1,\tau^2, \tau^3 \}$ may simply be taken to be given by the Pauli matrices $\{\sigma_x, \sigma_y, \sigma_z \}$ :
\begin{align}
	\sigma_x&= \begin{pmatrix}
		0 & 1 \\
		1 & 0
	\end{pmatrix}, \quad
	.
	\sigma_y = \begin{pmatrix}
		0 & -i \\
		i & 0
	\end{pmatrix}, \quad
	\sigma_z = \begin{pmatrix}
		1 & 0 \\
		0 & -1
	\end{pmatrix}.
\end{align}
Each gauge field  $W_mu:\mathds{R}^{1+1} \rightarrow \mathfrak{su}(2) $	may thus be written as \(W_\mu = W_\mu^b \frac{\tau^b}{2}\), with the functions \(W_\mu^c\) representing the components of the fixed background gauge field $W_\mu$.

Crucially, in the weak sector of the standard model, the gauge coupling acts only on the left-chiral components of the spinor field \cite{peskin1995introduction}. The corresponding covariant derivative \(D_\mu\) acting on each left-handed component is given as \cite{srednicki2007quantum}
\begin{align}
	D_\mu \psi_L^{(a)} := \partial_\mu \psi_L^{(a)} - i g \sum_{b=1}^2 \sum_{c=1}^3 \left( \frac{\tau^c}{2} \right)_{ab} W_\mu^c \psi_L^{(b)},
\end{align}
and on the right-handed components by the ordinary derivative
\begin{align}
	D_\mu \psi_R^{(a)} := \partial_\mu \psi_R^{(a)}.
\end{align}

Using the chiral (Weyl) representation of the Dirac gamma matrices in \(1+1\) dimensions $
\gamma^\mu = (\gamma^0, \gamma^1), 
\gamma^0 = \sigma^1, \quad
\gamma^1 = -i \sigma^2,
$
where \(\sigma^i\) are the Pauli matrices,
the Dirac equation is given as
\begin{align}
	(i \gamma^\mu D_\mu - m) \Psi = 0.
\end{align}

Unpacking this equation for each component \(a = 1,2\), and separating left- and right-handed parts, we obtain the two-component systems
\begin{align}
	i (\partial_t + \partial_x) \psi_R^{(a)} &= m \psi_L^{(a)}, \\
	i (\partial_t - \partial_x) \psi_L^{(a)} &= g \sum_{b=1}^2 \sum_{c=1}^3 \left( \frac{\tau^c}{2} \right)_{ab} (W_0^c - W_1^c) \psi_L^{(b)} + m \psi_R^{(a)}.
\end{align}

This system represents an evolution equation for the pair \((\psi_L^{(a)}, \psi_R^{(a)})\) with non-abelian coupling on the left-handed part and mass mixing between left- and right- chiralities. The gauge field couples the left-handed components across the \(\mathrm{SU}(2)\) indices \(a=1,2\), while the right-handed components evolve freely except for mass coupling.
Collecting the spinor doublets into 4-component vectors
\[
\psi_L := \begin{pmatrix} \psi_L^{(1)} \ \psi_L^{(2)} \end{pmatrix}^\top \quad \text{and} \quad
\psi_R := \begin{pmatrix} \psi_R^{(1)} \ \psi_R^{(2)} \end{pmatrix}^\top,
\]	
the coupled system takes the matrix form
\begin{align}
	i \frac{\partial}{\partial t} 
	\begin{pmatrix}
		\psi_L \\
		\psi_R
	\end{pmatrix}
	=
	\begin{pmatrix}
		- i \partial_x  \mathbb{I}_2 + g (W_0^c - W_1^c)  \frac{\tau^c}{2}
		& m \mathbb{I}_2 \\
		m \mathbb{I}_2 & i \partial_x  \mathbb{I}_2
	\end{pmatrix}
	\begin{pmatrix}
		\psi_L \\
		\psi_R
	\end{pmatrix}.
\end{align}

\subsection{1 + 3 Dimensions}\label{real_world_weak_example}

We now extend the previous analysis to spinor fields on four-dimensional Minkowski space \(\mathbb{R}^{1+3}\), interacting with a fixed, classical \(\mathrm{SU}(2)\) gauge field. As in the \(1+1\) case, the unknown field \(\Psi: \mathbb{R}^{1+3} \to \mathbb{C}^8\) is a doublet of Dirac spinors:
\[
\Psi = \begin{pmatrix} \psi^{(1)} \\ \psi^{(2)} \end{pmatrix}, \quad \psi^{(a)} = \begin{pmatrix} \psi_L^{(a)} \\ \psi_R^{(a)} \end{pmatrix}, \quad a = 1, 2,
\]
where \(\psi_L^{(a)}\) and \(\psi_R^{(a)}\) are now two-component Weyl spinors.

The \(\mathrm{SU}(2)\) gauge field is specified by \(W_\mu = W_\mu^b \, \tau^b / 2\), where \(\tau^1, \tau^2, \tau^3\) are the Pauli matrices. As before, the gauge coupling acts only on the left-chiral components, so the covariant derivative acts on left-handed components as:
\[
D_\mu \psi_L^{(a)} := \partial_\mu \psi_L^{(a)} - i g \sum_{b=1}^2 \sum_{c=1}^3 \left( \frac{\tau^c}{2} \right)_{ab} W_\mu^c \psi_L^{(b)},
\]
and on right-handed components by the ordinary derivative:
\[
D_\mu \psi_R^{(a)} := \partial_\mu \psi_R^{(a)}.
\]

We adopt the Weyl (chiral) representation for the Dirac gamma matrices in \(1+3\) dimensions:
\[
\gamma^\mu = \begin{pmatrix} 0 & \sigma^\mu \\ \bar{\sigma}^\mu & 0 \end{pmatrix}, \quad \text{with } \sigma^\mu = (I_2, \sigma^1, \sigma^2, \sigma^3), \quad \bar{\sigma}^\mu = (I_2, -\sigma^1, -\sigma^2, -\sigma^3),
\]
where the \(\sigma^i\) are the Pauli matrices.

The Dirac equation in this chiral basis takes the form:
\[
(i \gamma^\mu D_\mu - m) \Psi = 0.
\]

Expanding by components, we obtain:
\begin{align}
	i \sigma^\mu \partial_\mu \psi_R^{(a)} &= m \psi_L^{(a)}, \\
	i \bar{\sigma}^\mu \left( \partial_\mu \psi_L^{(a)} - i g \sum_{b=1}^2 \sum_{c=1}^3 \left( \frac{\tau^c}{2} \right)_{ab} W_\mu^c \psi_L^{(b)} \right) &= m \psi_R^{(a)}.
\end{align}

Rewriting in Hamiltonian (time-evolution) form:
\begin{align}
	i \partial_t \psi_L^{(a)} &= -i \sigma^i \left( \partial_i \delta_{ab} - i g \sum_{c=1}^3 \left( \frac{\tau^c}{2} \right)_{ab} W_i^c \right) \psi_L^{(b)} + g \sum_{c=1}^3 \left( \frac{\tau^c}{2} \right)_{ab} W_0^c \psi_L^{(b)} + m \psi_R^{(a)}, \\
	i \partial_t \psi_R^{(a)} &= -i \sigma^i \partial_i \psi_R^{(a)} + m \psi_L^{(a)}.
\end{align}

This system describes a coupled evolution of left- and right-chiral spinors, where only the left-handed components interact with the gauge field. As in the \(1+1\) case, the mass term introduces a chiral mixing, but the gauge field couples only the left-handed components across the \(\mathrm{SU}(2)\) indices \(a = 1, 2\).

We now collect the spinor doublets into four-component vectors:
\[
\psi_L := \begin{pmatrix} \psi_L^{(1)} \\ \psi_L^{(2)} \end{pmatrix}, \quad \psi_R := \begin{pmatrix} \psi_R^{(1)} \\ \psi_R^{(2)} \end{pmatrix},
\]
so that the coupled evolution system becomes:
\[
i \partial_t \begin{pmatrix} \psi_L \\ \psi_R \end{pmatrix} =
\begin{pmatrix}
	- i \sigma^i \left( \partial_i \otimes \mathbb{I}_2 - i g \sum_{c=1}^3 W_i^c \otimes \frac{\tau^c}{2} \right) + g \sum_{c=1}^3 W_0^c \otimes \frac{\tau^c}{2}
	& m \mathbb{I}_4 \\
	m \mathbb{I}_4 & -i \sigma^i \partial_i \otimes \mathbb{I}_2
\end{pmatrix}
\begin{pmatrix} \psi_L \\ \psi_R \end{pmatrix}
\]

This matrix equation has the same structure as its lower-dimensional analog, but now reflects the full spatial complexity of three dimensions. The system remains first-order in both time and space, with matrix-valued coefficients arising from the fixed non-Abelian background gauge field.

	\paragraph{Example:}	
	As an example for the setting of Theorem~\ref{convergence_to_compression}, let us consider the Hamiltonian operator that governs the time evolution of spinor fields subject to a chiral \(\mathrm{SU}(2)\) gauge interaction. This type of structure arises in the electroweak sector of the Standard Model of particle physics, where certain spin-\(\tfrac{1}{2}\) particles (such as electrons and neutrinos) are grouped into pairs called \(\mathrm{SU}(2)\) doublets. These doublets interact with a background gauge field through their left-handed components only. A spinor is said to be \emph{left-handed} or \emph{right-handed} based on how it transforms under spatial rotations and Lorentz boosts: roughly, these terms correspond to projections onto eigenspaces of the Dirac operator under a so called chirality operator, 
	separating the spinor into two components that evolve differently under relativistic symmetry transformations. The term “chiral coupling” refers to the fact that in the standard model only the left-handed component of each spinor field interacts with the gauge field, while the right-handed component evolves independently except for mass coupling.
	
	In the setting we consider, the gauge field is treated as a smooth, matrix-valued, externally prescribed background---that is, a fixed coefficient field entering the differential operator, not subject to any dynamical evolution. The resulting system describes the evolution of a pair of coupled spinor fields via a Dirac-type equation with internal matrix structure. The evolution equation takes the form
	\begin{align}
		i \frac{\partial}{\partial t} 
		\begin{pmatrix}
			\psi_L \\
			\psi_R
		\end{pmatrix}
		&=
		H_\beta
		\begin{pmatrix}
			\psi_L \\
			\psi_R
		\end{pmatrix},
	\end{align}
	with  the coupling constant $\beta$ dependent Hamiltonian operator
	\begin{align}
		A_\beta = 	\begin{pmatrix}
			- i \sigma^i \left( \partial_i \otimes \mathbb{I}_2 - i \beta \sum_{c=1}^3 W_i^c \otimes \frac{\tau^c}{2} \right) 
			+ \beta \sum_{c=1}^3 W_0^c \otimes \frac{\tau^c}{2}
			& m \mathbb{I}_4 \\
			m \mathbb{I}_4 & -i \sigma^i \partial_i \otimes \mathbb{I}_2
		\end{pmatrix}
	\end{align}
	
	Here  \(\psi_L\) and \(\psi_R\) are column vectors of two-component spinor fields representing the left- and right-chiral components of an \(\mathrm{SU}(2)\) spinor doublet. The functions \(W_\mu^c\) represent the components of the fixed background gauge field, and the \(\tau^c\) are the standard Pauli matrices generating the Lie algebra \(\mathfrak{su}(2)\). The mass term \(m\) couples the left- and right-handed fields. 
	
	To ilustrate Theorem~\ref{convergence_to_compression}, we now take each $W^c_i \in C^\infty(\mathds{R})\cap L^\infty(\mathds{R})$. Then it is easy to check that for any value of the coupling constant $\beta$ the Hamiltonian $H_\beta$ is self adjoint on the dense domain $\mathds{C}^8 \otimes H^{1}(\mathds{R})^3 \subseteq \mathds{C}^8 \otimes L^{2}(\mathds{R})^3$.
	
	For simplicity in presentation (similar results hold in the general case), let us assume that besides $W^3_0$ all other gauge fields vanish ($W_\mu^a =0$). Then we may write the coupling constant dependent Hamiltonian $H_\beta$ as $H_\beta = A + \beta B$, with
	\begin{align}
		A = 	\begin{pmatrix}
			- i \sigma^i  \partial_i \otimes \mathbb{I}_2 
			& m \mathbb{I}_4 \\
			m \mathbb{I}_4 & -i \sigma^i \partial_i \otimes \mathbb{I}_2
		\end{pmatrix}
		\quad \text{and} \quad
		B =\frac{1}{2} \begin{pmatrix}
			W_0^3 \otimes \sigma_3
			& 0\\
			0&  0
		\end{pmatrix}
		= \frac{W_0^3 }{2} \begin{pmatrix}
			I_2 \otimes \sigma_3 & 0 \\
			0 & 0
		\end{pmatrix}.
	\end{align}
	Assuming $W_0^3 (x) = 0$ holds only on a set of measure zero, the kernel of $B$ consists precisely of the right handed spinors $\psi_R \in \mathds{C}^4 \otimes L^2(\mathds{R}^3)$. Compressing $A$ onto the space of right handed spinors $\mathscr{V} \equiv 0 \oplus  \mathds{C}^4 \otimes L^2(\mathds{R}^3) \subseteq  \mathds{C}^8 \otimes L^2(\mathds{R}^3)$ yields $
	\underline{A} = -\sigma^i \otimes \mathbb{I}_2$,
	which is clearly self adjoint on its compressed domain $\mathds{C}^4 \otimes H^1(\mathds{R}^3)$.
	
	Thus we may apply Theorem~\ref{convergence_to_compression} to establish strong resolvent convergence of $H_\beta$ to the effective hamiltonian $\underline{A} =  -i\sigma^i \partial_i \otimes \mathbb{I}_2$. This limit Hamiltonian  describes a \textit{decoupled} model of two right handed  massless Weyl spinors. Each such spinor is respectively  governed by a Dirac equation.

\section{The 1+2 dimensional Dirac equation at fixed momentum}\label{fixed_momentum_de}
In this section we provide further context for Example \ref{high_energy_example_II}: Namely we consider the Dirac $2+1$-dimensional Dirac equation at fixed momentum in $y$ direction. In this setting, the full  Dirac equation in Hamiltonian form may be given as \cite{Thaller1992, Koke2020} 
\begin{align}
i \partial_t \Psi =  
		\left(
		- i \sigma_x \partial_x  - i \sigma_y \partial_y + m \sigma_z
		\right) \Psi.
\end{align}	

Restricting $\Psi$ to have fixed momentum $\Psi(x,y,t) = \tilde{\Psi}(x,t) \cdot e^{-k_y y}$ yields the effective equation
\begin{align}
	i \partial_t \tilde{\Psi} =  
	\left(
	- i \sigma_x \partial_x  -  \sigma_y k_y + m \sigma_z
	\right) \tilde{\Psi}.
\end{align}		
This is exactly the setting we consider in Example \ref{high_energy_example_II}, with the modification that instead of a single Dirac spinor we instead consider a doublet of them.

\section{Schur complement and block inverse calculations for the proof of Theorem \ref{thm_bounded_on_range}}\label{proof_support}
In Theorem \ref{thm_bounded_on_range} we consider the closed operator
	\begin{align}
	(A_\beta  - z)
	=
	\begin{pmatrix}
		[PA - zP]\restriction_{\mathcal{D}} & [PA - zP]\restriction_{\mathcal{D}^\intercal}   \\
		QA\restriction_{\mathcal{D}}    & [QA + \beta QB - zQ]\restriction_{\mathcal{D}^\intercal}
	\end{pmatrix},
\end{align}
where by assumption, the off-diagonal entries constitute bounded operators. 
We then defined the Schur complement	$ 	S := [QA + \beta QB - zQ - QAP R PA]\restriction_{\mathcal{D}^\intercal}  \in \mathcal{B}(\overline{\mathcal{D}^\intercal}^{\| \cdot \|_{\mathcal{H}}}, Q  \mathcal{H}),
$
and the resolvent 	$
R := (PA\restriction_{\mathcal{D}} - zP)^{-1} $ of $PA\restriction_{\mathcal{D}} - zP$ and verified that $R$ and $S^{-1}$ are well-defined (for sufficiently large $\beta \gg 1$).  We then defined a candidate inverse as:
\begin{align}
	T =
	\begin{bmatrix}
		R + R P[A\restriction_{\mathcal{D}^{\intercal}}] S^{-1} Q[A\restriction_{\mathcal{D}}]P R & -R P[A\restriction_{\mathcal{D}^{\intercal}}]S^{-1} \\
		- S^{-1} Q[A\restriction_{\mathcal{D}}] R & S^{-1}
	\end{bmatrix}.
\end{align}
Here we now verify that this is indeed a true inverse:
	
	\medskip
	
	We consider the operator
	\[
	A_\beta - z =
	\begin{bmatrix}
		A_{11} & A_{12} \\
		A_{21} & A_{22}
	\end{bmatrix}
	\]
	acting on the domain
	\[
	\operatorname{Dom}(A_\beta - z) = \mathcal{D} \oplus \mathcal{D}^\intercal \subset \mathcal{H},
	\]
	where the block operators are defined as follows with their respective domain and codomain:
	\begin{itemize}
		\item \( A_{11} := (PA - zP)\restriction_{\mathcal{D}} \colon \mathcal{D} \to P \mathcal{H} \),
		\item \( A_{12} := (PA - zP)\restriction_{\mathcal{D}^\intercal} \colon \mathcal{D}^\intercal \to P \mathcal{H} \),
		\item \( A_{21} := QA \restriction_{\mathcal{D}} \colon \mathcal{D} \to Q \mathcal{H} \),
		\item \( A_{22} := (QA + \beta QB - zQ)\restriction_{\mathcal{D}^\intercal} \colon \mathcal{D}^\intercal \to Q \mathcal{H} \).
	\end{itemize}
	
	The resolvent operator
	\[
	R := A_{11}^{-1} = (PA \restriction_{\mathcal{D}} - zP)^{-1} \colon P \mathcal{H} \to \mathcal{D}
	\]
	is well-defined and bounded, while the Schur complement
	\[
	S := A_{22} - A_{21} R A_{12} \colon \mathcal{D}^\intercal \to Q \mathcal{H}
	\]
	has a bounded inverse
	\[
	S^{-1} \colon Q \mathcal{H} \to \overline{\mathcal{D}^\intercal}^{\|\cdot\|_{\mathcal{H}}}.
	\]

	Furthermore domain and codomain  of each block entry of $T$ are as follows:
	
	\begin{itemize}
		\item \( T_{11} = R + R A_{12} S^{-1} A_{21} R \colon P \mathcal{H} \to \mathcal{D} \subset \mathcal{H} \), since
		\[
		R \colon P \mathcal{H} \to \mathcal{D}, \quad
		A_{21} R \colon P \mathcal{H} \to Q \mathcal{H}, \quad
		S^{-1} \colon Q \mathcal{H} \to \overline{\mathcal{D}^\intercal}, \quad
		A_{12} S^{-1} A_{21} R \colon P \mathcal{H} \to P \mathcal{H}.
		\]
		
		\item \( T_{12} = - R A_{12} S^{-1} \colon Q \mathcal{H} \to \mathcal{D} \subset \mathcal{H} \), since
		\[
		S^{-1} \colon Q \mathcal{H} \to \overline{\mathcal{D}^\intercal}, \quad
		A_{12} \colon \overline{\mathcal{D}^\intercal} \to P \mathcal{H}, \quad
		R \colon P \mathcal{H} \to \mathcal{D}.
		\]
		
		\item \( T_{21} = - S^{-1} A_{21} R \colon P \mathcal{H} \to \overline{\mathcal{D}^\intercal} \subset \mathcal{H} \), since
		\[
		R \colon P \mathcal{H} \to \mathcal{D}, \quad
		A_{21} \colon \mathcal{D} \to Q \mathcal{H}, \quad
		S^{-1} \colon Q \mathcal{H} \to \overline{\mathcal{D}^\intercal}.
		\]
		
		\item \( T_{22} = S^{-1} \colon Q \mathcal{H} \to \overline{\mathcal{D}^\intercal} \subset \mathcal{H}. \)
	\end{itemize}
	
	For ease of reference, this is summarized in the following table:
	
	\[
	\begin{array}{c|c|c}
		\text{Block} & \text{Domain} & \text{Codomain} \\
		\hline
		T_{11} & P \mathcal{H} & \mathcal{D} \subset \mathcal{H} \\
		T_{12} & Q \mathcal{H} & \mathcal{D} \subset \mathcal{H} \\
		T_{21} & P \mathcal{H} & \overline{\mathcal{D}^\intercal} \subset \mathcal{H} \\
		T_{22} & Q \mathcal{H} & \overline{\mathcal{D}^\intercal} \subset \mathcal{H}
	\end{array}
	\]

	We now compute the matrix product \( T (A_\beta - z) \) using block matrix multiplication:
	\begin{align}
		T (A_\beta - z) =
		\begin{bmatrix}
			T_{11} A_{11} + T_{12} A_{21} & T_{11} A_{12} + T_{12} A_{22} \\
			T_{21} A_{11} + T_{22} A_{21} & T_{21} A_{12} + T_{22} A_{22}
		\end{bmatrix},
	\end{align}
	where
	\begin{align}
		T_{11} &= R + R A_{12} S^{-1} A_{21} R, & T_{12} &= - R A_{12} S^{-1}, \\
		T_{21} &= - S^{-1} A_{21} R, & T_{22} &= S^{-1}.
	\end{align}
	
	We now verify each block entry:
	
	\paragraph{Top-left block:}
	\begin{align}
		T_{11} A_{11} + T_{12} A_{21}
		&= \left(R + R A_{12} S^{-1} A_{21} R\right) A_{11} - R A_{12} S^{-1} A_{21} \notag \\
		&= R A_{11} + R A_{12} S^{-1} A_{21} R A_{11} - R A_{12} S^{-1} A_{21} \notag \\
		&= I + R A_{12} S^{-1} A_{21} - R A_{12} S^{-1} A_{21} = I.
	\end{align}
	
	\paragraph{Top-right block:}
	\begin{align}
		T_{11} A_{12} + T_{12} A_{22}
		&= (R + R A_{12} S^{-1} A_{21} R) A_{12} - R A_{12} S^{-1} A_{22} \notag \\
		&= R A_{12} + R A_{12} S^{-1} A_{21} R A_{12} - R A_{12} S^{-1} A_{22} \notag \\
		&= R A_{12} \left( I + S^{-1} A_{21} R A_{12} - S^{-1} A_{22} \right).
	\end{align}
	Using the identity \( S = A_{22} - A_{21} R A_{12} \Rightarrow S^{-1} A_{22} = I + S^{-1} A_{21} R A_{12} \), we find that
	\begin{align}
		I + S^{-1} A_{21} R A_{12} - S^{-1} A_{22} = 0,
	\end{align}
	so the top-right block is zero.
	
	\paragraph{Bottom-left block:}
	\begin{align}
		T_{21} A_{11} + T_{22} A_{21}
		= -S^{-1} A_{21} R A_{11} + S^{-1} A_{21} = -S^{-1} A_{21} + S^{-1} A_{21} = 0.
	\end{align}
	
	\paragraph{Bottom-right block:}
	\begin{align}
		T_{21} A_{12} + T_{22} A_{22}
		= -S^{-1} A_{21} R A_{12} + S^{-1} A_{22} = S^{-1}(A_{22} - A_{21} R A_{12}) = S^{-1} S = I.
	\end{align}
	
	\medskip

	Let us also explicitly verify that $T$ is a left-inverse:
	
	 Using the same block notation, we compute:
	\begin{align}
		(A_\beta - z) T =
		\begin{bmatrix}
			A_{11} T_{11} + A_{12} T_{21} & A_{11} T_{12} + A_{12} T_{22} \\
			A_{21} T_{11} + A_{22} T_{21} & A_{21} T_{12} + A_{22} T_{22}
		\end{bmatrix}.
	\end{align}
	
	\paragraph{Top-left block:}
	\begin{align}
		A_{11} T_{11} + A_{12} T_{21}
		&= A_{11}(R + R A_{12} S^{-1} A_{21} R) + A_{12} (-S^{-1} A_{21} R) \notag \\
		&= A_{11} R + A_{11} R A_{12} S^{-1} A_{21} R - A_{12} S^{-1} A_{21} R \notag \\
		&= I + A_{12} S^{-1} A_{21} R - A_{12} S^{-1} A_{21} R = I.
	\end{align}
	
	\paragraph{Top-right block:}
	\begin{align}
		A_{11} T_{12} + A_{12} T_{22}
		&= A_{11}(- R A_{12} S^{-1}) + A_{12} S^{-1} \notag \\
		&= -A_{11} R A_{12} S^{-1} + A_{12} S^{-1} \notag \\
		&= - A_{12} S^{-1} + A_{12} S^{-1} = 0.
	\end{align}
	
	\paragraph{Bottom-left block:}
	\begin{align}
		A_{21} T_{11} + A_{22} T_{21}
		&= A_{21}(R + R A_{12} S^{-1} A_{21} R) + A_{22}(-S^{-1} A_{21} R) \notag \\
		&= A_{21} R + A_{21} R A_{12} S^{-1} A_{21} R - A_{22} S^{-1} A_{21} R \notag \\
		&= A_{21} R \left( I + A_{12} S^{-1} A_{21} R - S^{-1} A_{22} \right).
	\end{align}
	Using again the identity
	\begin{align}
		S^{-1} A_{22} = I + S^{-1} A_{21} R A_{12}
		\Rightarrow I + A_{12} S^{-1} A_{21} R - S^{-1} A_{22} = 0,
	\end{align}
	we find the bottom-left block is zero.
	
	\paragraph{Bottom-right block:}
	\begin{align}
		A_{21} T_{12} + A_{22} T_{22}
		&= A_{21}(- R A_{12} S^{-1}) + A_{22} S^{-1} \notag \\
		&= -A_{21} R A_{12} S^{-1} + A_{22} S^{-1} = S S^{-1} = I.
	\end{align}
	
	\medskip
	
	Hence, $
		(A_\beta - z) T = I$.
	This completes the verification that \( T \) is both a left and right inverse of \( A_\beta - z \), and thus, $
		T = (A_\beta - z)^{-1}$.

	\bibliographystyle{unsrt}
	\bibliography{bibfile}
\end{document}